\documentclass{article}%
\usepackage{amsmath}
\usepackage{amssymb}
\usepackage{amsmath}
\usepackage{amsfonts}
\usepackage{graphicx}
\usepackage[latin1]{inputenc}
\usepackage[english]{babel}
\usepackage{dsfont}
\usepackage{hyperref}
\usepackage{url}
\usepackage{abstract} 
\usepackage{stmaryrd}
\usepackage{authblk}  

\usepackage{vmargin}
\setmarginsrb{2.5cm}{2.5cm}{2.5cm}{2.5cm}{0cm}{0cm}{1.25cm}{1.25cm}
\providecommand{\U}[1]{\protect\rule{.1in}{.1in}}
\newtheorem{theorem}{Theorem}

\newtheorem{corollary}[theorem]{Corollary}

\newtheorem{definition}[theorem]{Definition}

\newtheorem{lemma}[theorem]{Lemma}

\newtheorem{proposition}[theorem]{Proposition}
\newtheorem{remark}[theorem]{Remark}

\newenvironment{proof}[1][Proof]{\noindent\textbf{#1.} }{\ \rule{0.5em}{0.5em}}
\begin{document}
	\title{$\mathbb{L}^p$-solutions $(1 <p< 2)$ for reflected BSDEs with general jumps and stochastic monotone generators}
	
	%
	
	\author[1]{Badr Elmansouri\thanks{Corresponding author.~This author was supported by the National Center for Scientific and Technical Research (CNRST), Morocco.}%
		\thanks{\href{https://orcid.org/0000-0003-2603-894X}{ORCID: 0000-0003-2603-894X}}%
		\thanks{Email: \href{mailto:badr.elmansouri@edu.uiz.ac.ma}{badr.elmansouri@edu.uiz.ac.ma}}}

	\author[1]{Mohamed El Otmani\thanks{\href{https://orcid.org/0000-0002-9674-6469}{ORCID: 0000-0002-9674-6469}}%
		\thanks{Email: \href{mailto:m.elotmani@uiz.ac.ma}{m.elotmani@uiz.ac.ma}}}
	
	\author[2]{Mohamed Marzougue\thanks{Email: \href{mailto:m.marzougue@uae.ac.ma}{m.marzougue@uae.ac.ma}}}
	
	\affil[1]{LAMA Laboratory, Faculty of Sciences Agadir, Ibn Zohr University,
		Hay Dakhla, BP8106, Agadir, Morocco.}
	
	\affil[2]{LAR2A Laboratory, Faculty of Sciences Tetouan,\\
		Abdelmalek Essaadi University, 93000 Tetouan, Morocco.}
	
	%
	%
	\date{}
	\maketitle
	\begin{abstract}
		We consider a one-reflected backward stochastic differential equation with a general RCLL barrier in a filtration that supports a Brownian motion and an independent Poisson random measure. We establish the existence and uniqueness of a solution in $\mathbb{L}^p$ for $p \in (1,2)$. The result is obtained by means of the penalization method, under the assumption that the coefficient is stochastically monotone with respect to the state variable $y$, stochastically Lipschitz with respect to the control variables $(z,u)$, and satisfies suitable linear growth and $p$-integrability conditions.
		\vspace{0.3cm}
		
		
	\end{abstract}

	\section{Introduction}
Non-linear backward stochastic differential equations (BSDEs in short) were initially introduced by Pardoux and Peng \cite{PP}. Precisely, given a data $(\xi,f)$ of a square integrable random variable $\xi$  and a progressively measurable function $f$, a solution of BSDEs associated with data $(\xi,f)$ is a pair of $\mathcal{F}_t$-adapted processes $(Y,Z)$ satisfying
\begin{equation*}
	Y_{t}=\xi +\int_{t}^{T} f(s,Y_{s},Z_{s})ds-\int_{t}^{T}Z_{s}dB_{s},\quad 0\leq t\leq T,
\end{equation*}
where $B$ is a standard Brownian motion. Since Pardoux and Peng have established the existence and uniqueness of $L^2$ solutions for multi-dimensional BSDEs, the theory of BSDEs has been intensively developed in the last years, the great interest in this theory comes from its connections with many other fields of research, such as mathematical finance \cite{EPeQ,EQ}, stochastic control and stochastic games \cite{EH} and partial differential equations \cite{PP92}. In particular, many efforts have been devoted to existence and uniqueness results under weaker assumptions on the data. In this context, some works were devoted to weaken the square integrability on the coefficient and the terminal value. We are talking about $\mathbb{L}^p$ solutions where $p \in (1,2)$. In general, the $\mathbb{L}^p$ ($p>1$) solutions of BSDEs were investigated, at the first time, by El Karoui et al. \cite{EPeQ} under the Lipschitz coefficient. After that, Briand et al. \cite{BDHPS} have proved an existence result of $L^p$ ($p>1$) solutions for multi-dimensional BSDEs where $f$ satisfies a kind of monotonicity condition with a general growth in $y$. Since then, many researchers have continued to develop an array of further researches. In this context, Briand et al. \cite{BLS} have established an existence result of $\mathbb{L}^p$ ($p>1$) solutions for one-dimensional BSDEs where the generator $f$ satisfies the monotonicity condition in $y$ employed in Briand et al. \cite{BDHPS} and a linear growth condition in $z$, Chen \cite{Ch} have proved the existence and uniqueness of $\mathbb{L}^p$ ($p \in (1, 2]$) solutions for one-dimensional BSDEs when the generator $f$ satisfies a uniform continuity condition in $(y,z)$, Ma et al. \cite{MFS} have investigated the existence and uniqueness result for $\mathbb{L}^p$ ($p> 1$) solutions of one-dimensional BSDEs when the generator is monotonic with a general growth in $y$, and uniformly continuous in $z$, Tian et al. \cite{TJS} have considered a one-dimensional BSDEs and proved an existence theorem of $L^p$ ($p>1$) solutions whose generator satisfies a kind of discontinuous condition in $y$ and is uniformly continuous in $z$. Particularly, Fan \cite{Fan2015} have obtained the existence and uniqueness of $\mathbb{L}^p$ ($p>1$) solutions for multi-dimensional BSDEs under a ($p\wedge2$)-order weak monotonicity condition with a general growth in $y$ and a Lipschitz condition in $z$ for the generator. Recently, Wang et al. \cite{WLF} have established the existence of a minimal (maximal) $\mathbb{L}^p$ ($p \in (1,2]$) solution to a one-dimensional BSDEs, where the generator satisfies a $p$-order weak monotonicity condition together with a general growth condition in $y$ and a linear growth condition in $z$.

We especially mention that El Karoui and Huang \cite{ELH} are the first which considered the so-called stochastic Lipschitz coefficient, where the Lipschitz constant is replaced by a stochastic process. They authors have considered a general time interval BSDEs driven by a general RCLL martingale, and some stronger integrability conditions on the generator and terminal value. In this spirit, Bender and Kohlmann \cite{BK} have proved an existence and uniqueness result for the $L^2$ solution. After that, Wang et al. \cite{WRC} have studied the existence and uniqueness of $\mathbb{L}^p$ ($p>1$) solution for BSDEs.

Tang and Li \cite{TL} have introduced into the BSDEs a jumps term (BSDEJs in short) that is driven by a Poisson random measure independent of the Brownian motion. The authors have proved the existence of a unique solution with a Lipschitz generator and square integrable terminal value. A solution of BSDEJs associated with data $(\xi,f)$ is a triplet of $\mathbb{F}$-adapted processes $(Y,Z,U)$ satisfying
\begin{equation*}
	Y_{t}=\xi+\int_{t}^{T} f(s,Y_{s},Z_{s},U_s)ds-\int_{t}^{T}Z_{s}dB_{s}-\int_t^T\int_{\mathcal{U}}U_s(e)\tilde{\mu}(ds,de),\quad 0\leq t\leq T,
\end{equation*}
where $\tilde{\mu}$ is a compensated Poisson random measure. Then, Barles et al. \cite{BBP} have showed that the well-posedness of BSDEJs gives rise to a viscosity solution of a semi-linear parabolic partial integro-differential equation. Later, Rong \cite{Ron} degenerated the monotonicity condition of the generator to a weaker version so as to remove the Lipschitz condition on variable $z$. Now, regarding the case of $L^p$ solutions, Kruse and Popier \cite{Kruse17112017,Kruse18052016} have analyzed multi-dimensional BSDEs in a general filtration that supports a Brownian motion and a Poisson random measure. The authors have established the existence and uniqueness of solutions under monotonicity assumption on the driver in $\mathbb{L}^p$ space provided that the generator and the terminal value satisfy appropriate integrability conditions. Recently, Yao \cite{Yao} has studied $\mathbb{L}^p$ ($p \in (1,2)$) solutions of a multi-dimensional BSDEJs whose generator may not be Lipschitz continuous in $(y,z)$. The author has showed the existence and uniqueness of solution by approximating the monotonic generator by a sequence of Lipschitz generators.

El Karoui et al. \cite{EKPPQ} have introduced the notion of reflected BSDEs (RBSDEs in short), which is a BSDEs but the solution is forced to stay above a given process called barrier (or obstacle). Once more under square integrability of the terminal condition and the barrier and Lipschitz property of the coefficient, the authors have proved the existence and uniqueness results in the case of a Brownian filtration and a continuous barrier. Many efforts have been devoted to existence and uniqueness results under weaker assumptions, in this context, Hamad\`{e}ne \cite{H} has considered the RBSDEs when the barrier is discontinuous, Hamad\`{e}ne and Ouknine \cite{HO} have studied the RBSDEs when the noise is driven by a Brownian motion and an independent Poisson measure (that is the generalization of the work of El Karoui et al. \cite{EKPPQ}) and after RBSDEs with "only" \textit{rcll} obstacle in \cite{Ess,HO1}. On the other hand, Hamadène and Popier \cite{HP} have studied the existence and uniqueness for $L^p$ ($1<p<2$) solutions in a Brownian framework and Lipschitz coefficient, the same frame to Hamadène and Popier \cite{HP} has been treated by Rozkosz and Slominski \cite{RS} but the generator satisfying the monotonicity condition. Recently, Yao \cite{Yao1} has considered RBSDEs with jumps (RBSDEJs in short) whose generator is Lipschitz continuous in $(y,z,v)$, the author has established the existence and uniqueness of a $\mathbb{L}^p$ ($p \in (1,2)$) solutions of such equations via a fixed point theorem. There are many recent developments on $L^p$ solutions of RBSDEJs in various interesting directions, namely Klimsiak \cite{K,Kl}, Eddahbi et al. \cite{EFO}, El Jamali \cite{ElJamali}, Elmansouri and Marzougue \cite{ELMANSOURI2025110407}, Li and Wei \cite{LI20141582} and so on. and so on.

The main objective of this paper is to continue the developments mentioned above and to study one-dimensional reflected backward stochastic differential equations with jumps (RBSDEJs), where the obstacle is only assumed to be RCLL with general jumps, and the noise is driven by a Brownian motion and an independent Poisson random measure. We focus on the case where the terminal condition and the generator are only $p$-integrable for some $p \in (1,2)$. Moreover, we are interested in the case where the generator satisfies a \textit{stochastic monotonicity condition} with respect to the state variable $y$, and a \textit{stochastic Lipschitz condition} with respect to the control variables $(z,u)$, along with suitable linear growth assumptions. It is worth noting that the stochastic Lipschitz (or monotonic) condition arises naturally in many applications (e.g., in finance as in \cite{EMS,Marzougue2020}), where the usual Lipschitz condition is too restrictive to be satisfied. Our goal is to establish existence and uniqueness of solutions for such equations using \textit{penalization methods} combined with the \textit{Banach fixed-point theorem}. Additionally, compared to the existing literature on $\mathbb{L}^p$-solutions for $p \in (1,2)$, our work improves upon all known results involving stochastic monotonicity, including the recent works of Elmansouri and El Otmani \cite{elmansouri2025,badrGBSDE}, and Li and Fan \cite{li2024weightedlppgeq1solutionsrandom}.

The paper is organized as follows. In Section \ref{s2}, we introduce the notations, assumptions, and preliminary results required throughout the paper, and we define the class of RBSDEJs under consideration. Section \ref{s4} is devoted to establishing the existence and uniqueness of $\mathbb{L}^p$-solutions using a penalization method.

\section{Preliminaries and notations}\label{s2}
Let $T > 0$ be a fixed time horizon and let $p \in (1,2)$. Consider a filtered probability space $(\Omega, \mathcal{F},\mathbb{F} := (\mathcal{F}_t)_{t \leq T}, \mathbb{P})$, where the filtration $\mathbb{F}$ is assumed to be complete and right-continuous. We further assume that $\mathbb{F}$ is generated by a $d$-dimensional Brownian motion $(B_t)_{t \leq T}$ and an independent compensated Poisson random measure $\tilde{\mu}$, associated with a standard Poisson random measure $\mu$ defined on $\mathbb{R}^+ \times \mathcal{U}$. Here, $\mathcal{U} := \mathbb{R}^d \setminus \{0\}$ (with $d \geq 1$) is equipped with its Borel $\sigma$-algebra $\mathbb{U}$, and the compensator of $\mu$ is given by $\nu(dt, de) = dt \lambda(de)$. The process $\{\tilde{\mu}([0,t] \times \mathcal{G}) = (\mu - \nu)([0,t] \times \mathcal{G})\}_{t \leq T}$ defines a martingale for every measurable set $\mathcal{G} \in \mathbb{U}$ such that $\lambda(\mathcal{G}) < +\infty$. The compensator measure $\lambda$ is assumed to be a $\sigma$-finite measure on $\mathcal{U}$ satisfying the integrability condition
$$
\int_{\mathcal{U}} (1 \wedge |e|^2) \lambda(de) < +\infty.
$$

We adopt the following notational conventions:
\begin{itemize}
	\item $|.|$ denotes the Euclidean norm in $\mathbb{R}^{d}$ for $d \geq 1$.
	\item For all $x \in \mathbb{R}$, we denote $x^+ := \max(x, 0)$ and $x^- := -\min(x, 0)$.
	\item $\mathcal{T}_{[t,T]}$ refers to the collection of stopping times $\tau$ such that $\tau \in [t,T]$, for every fixed $t \in [0,T]$.
	\item $\mathcal{T}^p_{[t,T]}$ denotes the set of $\mathbb{F}$-predictable stopping times $\tau$ satisfying $\tau \in [t,T]$, for any given $t \in [0,T]$.
	\item $\mathcal{P}$ designates the predictable $\sigma$-algebra on the product space $\Omega \times [0,T]$.
	\item $\mathcal{B}(\mathbb{R}^{d})$ represents the Borel $\sigma$-algebra on $\mathbb{R}^{d}$.
\end{itemize}

Given a right-continuous process with left limits paths (RCLL) $(\varpi_t)_{t \leq T}$, we define the left limit of $\varpi$ at any time $t \in (0,T]$ by $\varpi_{t-} := \lim\limits_{s \nearrow t} \varpi_s$, with the convention that $\varpi_{0-} := \varpi_0$. The left-limit process is denoted by $\varpi_- := (\varpi_{t-})_{t \leq T}$. The jump of $\varpi$ at time $t \in [0,T]$ is given by $\Delta \varpi_t := \varpi_t - \varpi_{t-}$.  For a finite variation, RCLL, $\mathbb{R}$-valued process ${K}$, we denote by ${K}^c$ and ${K}^d$ the continuous and purely discontinuous parts of ${K}$, respectively, satisfying the decomposition ${K} = {K}^c + {K}^d$ with ${K}^c_0={K}_0$ and ${K}^d = \sum_{0 < s \leq \cdot} \Delta {K}_s$. For notational convenience throughout the proof, we frequently denote by $\mathcal{X}_\ast := \sup_{t \in [0,T]} \mathcal{X}_t$ the essential supremum of any RCLL process $\mathcal{X} := (\mathcal{X}_t)_{t \leq T}$. 

Let $\beta > 0$ and $(a_{t})_{t\leq T}$ be a nonnegative $\mathcal{F}_{t}$-adapted process. We define the increasing continuous process $A_t:= \int_0^t \zeta^2_sds$, for all  $t \in [0,T]$, and we introduce the following spaces:
\begin{itemize}
	\item[$\bullet$] $\mathbb{L}^2_\lambda$ is the set of $\mathbb{R}$-valued and $\mathbb{U}$-measurable mapping $V:\mathcal{U}\rightarrow \mathbb{R}$ such that
	$$\|V\|_\lambda^{2}=\int_{\mathcal{U}}|V(e)|^2\lambda(de) < +\infty.$$
	\item[$\bullet$] $\mathcal{L}^{p}_\beta$ is the space of $\mathbb{R}$-valued and $\mathcal{F}_{T}$-measurable random variables $\xi$ such that
	$$\|\xi\|_{\mathcal{L}^{p}_\beta}=\left(\mathbb{E}\left[e^{\frac{p}{2}\beta A_T}|\xi|^{p}\right]\right)^\frac{1}{p} < +\infty.$$
	\item[$\bullet$] $\mathcal{K}^{p}$ is the space of $\mathbb{R}$-valued RCLL  $\mathbb{F}$-predictable increasing processes $K$ such that $K_0=0$
	$$\|K\|_{\mathcal{K}^{p}}=\left(\mathbb{E}\left[|K_T|^{p}\right]\right)^\frac{1}{p} < +\infty.$$
	\item[$\bullet$] $\mathcal{S}^{p}_\beta$ is the space of $\mathbb{R}$-valued and $\mathbb{F}$-adapted RCLL processes $(Y_{t})_{t\leq T}$ such that
	$$\|Y\|_{\mathcal{S}^{p}_\beta}=\left(\mathbb{E}\left[\sup\limits_{0\leq t\leq T}e^{\frac{p}{2}\beta A_t}|Y_{t}|^{p}\right]\right)^\frac{1}{p} < +\infty,$$
	with the convention $\mathcal{S}^p:=\mathcal{S}^p_0$.
	\item[$\bullet$] $\mathcal{S}^{p,A}_\beta$ is the space of $\mathbb{R}$-valued and $\mathbb{F}$-adapted RCLL processes $(Y_{t})_{t\leq T}$ such that
	$$\|Y\|_{\mathcal{S}^{p,A}_\beta}=\left(\mathbb{E}\left[\int_{0}^{T}e^{\frac{p}{2}\beta A_t}|Y_{t}|^{p}dA_t\right]\right)^\frac{1}{p} < +\infty.$$
	\item[$\bullet$] $\mathcal{H}^{p}_\beta$ is the space of $\mathbb{R}^{d}$-valued and $\mathbb{F}$-predictable processes $(Z_{t})_{t\leq T}$ such that
	$$\|Z\|_{\mathcal{H}^{p}_\beta}=\left(\mathbb{E}\left[\left(\int_{0}^{T}e^{\beta A_t}|Z_{t}|^{2}dt\right)^\frac{p}{2}\right]\right)^\frac{1}{p} < +\infty,$$
	with the convention $\mathcal{H}^p:=\mathcal{H}^p_0$.
	\item[$\bullet$] $\mathfrak{L}^{p}_{\lambda,\beta}$ is the space of $\mathbb{R}$-valued and $\mathcal{P}\otimes \mathbb{U}$-measurable processes $(U_t)_{t\leq T}$ such that
	$$\|U\|_{\mathfrak{L}^{p}_{\lambda,\beta}}=\left(\mathbb{E}\left[\left(\int_0^Te^{\beta A_t}\|U_t\|_\lambda^2dt\right)^\frac{p}{2}\right]\right)^\frac{1}{p} < +\infty,$$
	with the convention $\mathfrak{L}^{p}_{\lambda}:=\mathfrak{L}^{p}_{\lambda,0}$.
	\item[$\bullet$] $\mathfrak{L}^{p}_{\mu,\beta}$ is the space of $\mathbb{R}$-valued and $\mathcal{P}\otimes \mathbb{U}$-measurable processes $(U_t)_{t\leq T}$ such that
	$$\|U\|_{\mathfrak{L}^{p}_{\mu,\beta}}=\left(\mathbb{E}\left[\left(\int_0^T\int_{\mathcal{U}}e^{\beta A_t}|U_t(e)|^2\mu(dt,de)\right)^\frac{p}{2}\right]\right)^\frac{1}{p} < +\infty,$$
	with the convention $\mathfrak{L}^{p}_{\mu}:=\mathfrak{L}^{p}_{\mu,0}$.
	\item[$\bullet$] $\mathfrak{B}^p_\beta:=\mathcal{S}^{p}_\beta\cap \mathcal{S}^{p,A}_\beta$ is a Banach space endowed with the norm
	$\|Y\|_{\mathfrak{B}^{p}_\beta}^p=\|Y\|_{\mathcal{S}^{p}_\beta}^p+\|Y\|_{\mathcal{S}^{p,A}_\beta}^p.$
	\item[$\bullet$] $\mathcal{E}^p_\beta:=\mathfrak{B}^{p}_\beta \times \mathcal{H}^{p}_\beta \times \mathfrak{L}^{p}_{\lambda,\beta} \times \mathcal{K}^{p}.$
\end{itemize}

Throughout this paper, we denote by $\mathfrak{c}$ and $\mathfrak{C}$ generic constants whose values may vary from line to line. When a constant depends explicitly on a given set of parameters $\gamma$, we will write $\mathfrak{c}_\gamma$ (resp. $\mathfrak{C}_\gamma$) to highlight this dependence.

In this paper, we aim to find a quadruple of processes $(Y,Z,U,K)$ taking values in $\mathbb{R}^d \times \mathbb{R}^{d \times k} \times \mathfrak{L}^2_\lambda \times \mathbb{R}^d$ that satisfies the following RBSDEJ associated with the data $(\xi, f, L)$:
\begin{equation}\label{basic equation}
	\displaystyle\left\{
	\begin{split}
		&\text{(i)}~Y_{t} = \xi + \int_{t}^{T} f(s, Y_{s}, Z_{s}, U_s) \, ds + \left(K_{T} - K_{t}\right) - \int_{t}^{T} Z_{s} \, dB_{s} - \int_t^T \int_{\mathcal{U}} U_s(e) \, \tilde{\mu}(ds,de), \quad t \in [0,T], \\
		&\text{(ii)}~Y_t \geq L_t, \quad \forall t \in [0,T], \\
		&\text{(iii)}~\int_0^T (Y_t - L_t) \, dK^c_t = 0 \; \text{a.s.} \quad \text{and} \quad \Delta K_t^d = (Y_t - L_{t-})^- \, \mathds{1}_{\{Y_{t-} = L_{t-}\}} \; \text{a.s.}
	\end{split}
	\right.
\end{equation}

We seek such a quadruple in the space $\mathfrak{B}^{p}_\beta \times \mathcal{H}^{p}_\beta \times \mathfrak{L}^{p}_{\lambda,\beta} \times \mathcal{S}^{p}$ for some $p \in (1,2)$. Any such quadruple is referred to as an $\mathbb{L}^p$-solution of the RBSDEJ \eqref{basic equation}. More formally, we give the following definition:
\begin{definition}\label{Def}
	Let $p \in (1,2)$. A quadruple of processes $(Y,Z,U,K)$ with values in $\mathbb{R}^d \times \mathbb{R}^{d \times k} \times \mathbb{L}^2_\lambda \times \mathbb{R}^d$ is said to be an $\mathbb{L}^p$-solution of the RBSDEJ \eqref{basic equation} associated with the data $(\xi, f, L)$ if it satisfies equation \eqref{basic equation} and belongs to the space $\mathcal{E}^p_\beta$.
\end{definition}

\paragraph{Assumptions on the data $(\xi, f, L)$:}\emph{}\\
In what follows, we impose the following conditions on the terminal condition $\xi$, the generator function $f$, and the obstacle process $L$:\\ 
Let $p \in (1,2)$
\begin{description}
	\item[$(\mathcal{H}1)$] The terminal condition $\xi \in \mathcal{L}^{p}_\beta$.
	\item[$(\mathcal{H}2)$] The coefficient $f : \Omega\times [0,T]\times\mathbb{R}\times\mathbb{R}^d\times \mathbb{L}^2_\lambda\longrightarrow \mathbb{R}$ satisfies :
	\begin{description}
		\item[(i)] For all $(y,z,u)\in \mathbb{R} \times \mathbb{R}^{d}\times\mathbb{L}^2_\lambda$, the process $(f(t,y,z,u))_{t\leq T}$ is progressively measurable.
		
		\item[{(ii)}] There exists an $\mathbb{F}$-progressively measurable processes $\alpha : \Omega \times [0,T] \rightarrow \mathbb{R}$ such that for all $t \in [0,T]$, $y,y' \in \mathbb{R}$, $z \in \mathbb{R}^{d}$, $u \in \mathbb{L}^2_\lambda$, $d\mathbb{P}\otimes dt$-a.e.,
		$$
		\left(y-y'\right)\left(f(t,y,z,u)-f(t,y',z,u)\right) \leq \alpha_t \left|y-y'\right|^2.
		$$
		
		\item[{(iii)}] There exists two $\mathbb{F}$-progressively measurable processes $\eta, \delta : \Omega \times [0,T] \rightarrow \mathbb{R}_{+}$ and  such that for all $t \in [0,T]$, $y, \in \mathbb{R}$, $z,z' \in \mathbb{R}^{d}$, $u,u' \in \mathbb{L}^2_\lambda$, $d\mathbb{P}\otimes dt$-a.e.,
		$$
		\left| f(t,y,z,u)-f(t,y,z',u')\right|  \leq \eta_t \left|z-z'\right|+\delta_t \left\|u-u'\right\|_{\lambda}.
		$$
		
		\item[{(iv)}] There exist two progressively measurable processes $\varphi : \Omega \times \left[0,T\right] \rightarrow [1,+\infty)$, $\phi : \Omega \times \left[0,T\right] \rightarrow (0,+\infty)$ such that $\left|f(t,y,0,0)\right| \leq \varphi_t+\phi_t |y|$ and 
		$$
		\mathbb{E}\left[ {\int _{0}^{T}}{e^{\beta {A_{s}}}} \left|\varphi_s\right|^pds\right]  <+\infty ,
		$$
		
		\item[(v)] We set $a_s^2=\phi_s+\eta^2_s+\delta^2_s$ with $a_s \geq 0$ for any $s \in [0,T]$ and then we define  $\zeta_s^2=a_s^{q}$ with $q=\frac{p}{p-1}$ and we assume that there exists a constant $\epsilon>0$ such that $a^2_s \geq \epsilon$ for any $s \in [0,T]$. Note that $q$ corresponds to the conjugate exponent of $p$, i.e., $\frac{1}{p} + \frac{1}{q} = 1$, which will be useful when applying Hölder's inequality. Moreover, we have $\zeta_s \geq \epsilon^{\frac{q}{4}}$ for any $s \in [0,T]$.
		
		\item[{(vi)}] For all $(t, z, u) \in [0,T] \times \mathbb{R}^{d} \times \mathbb{L}^2_\lambda$, the mapping $y \mapsto f(t, y, z, u)$ is continuous, $\mathbb{P}$-a.s.
	\end{description}
	\item[$(\mathcal{H}3)$] The obstacle $(L_{t})_{t\leq T}$ is a RCLL progressively measurable real-valued process satisfying
	\begin{description}
		\item[$(i)$] $L_T\leq \xi$.
		\item[$(ii)$] $\mathbb{E}\left[\sup\limits_{0\leq t\leq T}\left|e^{\frac{q}{2}\beta A_t}L_t^+\right|^{p}\right]<+\infty$.
	\end{description}
\end{description}
\begin{remark}\label{rmq essential}
	To simplify computations and obtain optimal constants in the a priori estimates of the solutions, we shall assume throughout the remainder of this paper that condition $(\mathcal{H}2)$ holds with a process $(\alpha_t)_{t \leq T}$ satisfying $\alpha_t + \varepsilon a_t^2 \leq 0$ for every $\varepsilon \geq 0$ and all $t \in [0,T]$. In cases where this assumption does not hold, one may apply a change of variable analogous to the one described in \cite[Remark 5]{elmansouri2025} or \cite[Remark 3]{badrGBSDE}, suitably adapted to account for the reflecting process $K$ in the RBSDEJ \eqref{basic equation}, thereby reducing the setting to the required framework.
\end{remark}

\begin{remark}\label{Obse}
	It is important to emphasize, as noted in \cite{Kruse17112017}, that for $p \in (1,2)$, the compensator of a purely discontinuous martingale does not, in general, dominate its predictable projection (see \cite{Lenglart1980} for a counterexample). In this setting, for any $U \in \mathfrak{L}^p_\lambda$, one cannot control the quantity 
	$$
	\mathbb{E}\left[\left(\int_0^T e^{\beta A_t} \|U_t\|^2_\lambda \, dt\right)^{\frac{p}{2}}\right]
	$$
	by the term
	$$
	\mathbb{E}\left[\left(\int_0^T \int_{\mathcal{U}} e^{\beta A_t} |U_t(e)|^2 \, \mu(dt,de)\right)^{\frac{p}{2}}\right].
	$$
	However, exploiting the concavity of the function $x \mapsto x^{p/2}$ on $\mathbb{R}^+$ and applying Theorem 4.1-(2) in \cite{Lenglart1980}, one obtains that for every $U \in \mathfrak{L}^{p}_{\mu,\beta}$, the following inequality holds:
	\begin{equation*}\label{E}
		\mathbb{E}\left[\left(\int_{0}^{T} e^{\beta A_s} \int_{\mathcal{U}} \big|U_s(e)\big|^2 \, \mu(ds,de)\right)^{\frac{p}{2}}\right] 
		\leq 2 \, \mathbb{E}\left[\left(\int_{0}^{T} e^{\beta A_s} \|U_s\|^2_{\mathbb{L}^2_\lambda} \, ds\right)^{\frac{p}{2}}\right].
	\end{equation*}
	This leads to the inclusion $\mathfrak{L}^{p}_{\lambda,\beta} \subset \mathfrak{L}^{p}_{\mu,\beta}$ and $\mathfrak{L}^{p}_{\lambda,\beta}\cap\mathfrak{L}^{p}_{\mu,\beta}=\mathfrak{L}^{p}_{\lambda,\beta}$. 
	
	The situation is different in the case of continuous martingales such as $\int_0^{\cdot} Z_s \, dB_s$ for $Z \in \mathcal{H}^p_\beta$, where the classical Burkholder-Davis-Gundy (BDG, for short) inequalities remain valid (see \cite{Lenglart1980}, p. 37).
\end{remark}

We will provide a version of It\^{o}'s formula applied to the function $(t,x) \mapsto e^{\frac{p}{2}\beta A_t} |x|^p$ ($p \in (1,2)$), which is not sufficiently smooth. We define $\hat{x} := |x|^{-1}x\mathds{1}_{\{x \neq 0\}}$.\\
This result, which will be used multiple times, is a slight variation of \cite[Lemma 7]{Kruse18052016}, where we incorporate a predictable process of finite variation $(K_t)_{t \leq T}$ instead of an orthogonal martingale, as considered in the BSDE framework of \cite{Kruse18052016}. The case for a filtration generated by a Brownian motion is detailed in \cite[Lemma 2.2]{BRIAND2003109} and in the reflected case is presented in \cite[Corollary 1]{HamadenePopier2012}.
\begin{lemma}
	We consider the $\mathbb{R}$-valued semimartingale $(X_t)_{t\leq T}$ defined by
	$$ X_{t}=X_0+\int_{0}^{t}F_sds+\int_{0}^{t}Z_{s}dB_{s}+\int_0^t\int_{\mathcal{U}}U_s(e)\tilde{\mu}(ds,de)+K_t,$$
	such that:
	\begin{itemize}
		\item  $\mathbb{P}$-a.s. the process $K$ is predictable of bounded variation.
		
		\item $\left(F_t\right)_{t \leq T}$ is an $\mathbb{R}$-valued progressively measurable process and $\left(Z_t\right)_{t \leq T}$, $\left(U_t\right)_{t \leq T}$ are predictable processes with values in  $\mathbb{R}^d$, $\mathbb{L}^2_\lambda$, respectively, such that $\int_0^T\left\{F_t+|Z_t|^2+\|U_t\|_\lambda^2\right\}dt<+\infty$, $\mathbb{P}$-a.s.
	\end{itemize}
	Then, for any $p\geq 1$ there exists a continuous and non-decreasing process $(\ell_t)_{t\leq T}$ such that
	\begin{eqnarray}\label{e0}
		e^{\frac{p}{2}\beta A_t}|X_t|^p&=&|X_0|^p+\frac{p}{2}\beta\int_0^te^{\frac{p}{2}\beta A_s}|X_s|^pdA_s+\frac{1}{2}\int_0^te^{\frac{p}{2}\beta A_s}\mathds{1}_{\{p=1\}}d\ell_s+p\int_0^te^{\frac{p}{2}\beta A_s}|X_s|^{p-1}\hat{X}_sF_s ds\nonumber\\
		&&+p\int_0^te^{\frac{p}{2}\beta A_s}|X_s|^{p-1}\hat{X}_sZ_sdB_s+p\int_0^t\int_{\mathcal{U}}e^{\frac{p}{2}\beta A_s}|X_{s-}|^{p-1}\hat{X}_{s-}U_s(e)\tilde{\mu}(ds,de)\nonumber\\
		&&+p\int_0^te^{\frac{p}{2}\beta A_s}|X_{s-}|^{p-1}\hat{X}_{s-}dK_s+c(p)\int_0^te^{\frac{p}{2}\beta A_s}|X_s|^{p-2}|Z_s|^2
		\mathds{1}_{\{X_s\neq0\}}ds\nonumber\\
		&&+\int_0^t\int_{\mathcal{U}}e^{\frac{p}{2}\beta A_s}\left[|X_{s-}+U_s(e)|^{p}-|X_{s-}|^p-p|X_{s-}|^{p-1}\hat{X}_{s-}U_s(e)\right]\mu(ds,de),\nonumber
	\end{eqnarray}
	where $c(p):=\frac{p(p-1)}{2}$ and  $(\ell_t)_{t\leq T}$ is a continuous, non-decreasing process that increases only on the boundary of the random set $\{t\leq T,\;\; X_{t-}=X_t=0\}$.
\end{lemma}
\begin{proof}
	Here, we only sketch the proof and highlight the main differences.\\
	Using the pathwise decomposition of the process $(K_t)_{t \leq T}$ as $K_t = K^c_t + \sum_{0 < s \leq t} \Delta K_s$ and considering the function $\nu_\varepsilon : \mathbb{R} \rightarrow \mathbb{R}$ defined by $\nu_\varepsilon(x) = (|x|^2 + \varepsilon^2)^\frac{1}{2}$ for $\varepsilon > 0$ with its first derivative $\left(\nu_\varepsilon(x)^p\right)^\prime = p x \nu_\varepsilon(x)^{p-2}$, we obtain the following convergences:
	\begin{itemize}
		\item As shown in \cite[Lemma 2.2]{BRIAND2003109}, for the terms involving the first derivatives of $\nu_\varepsilon$ with respect to $K^c$, we have
		$$
		\int_{0}^{t} e^{\frac{p}{2}\beta A_s}\nu_\varepsilon(X_{s})^{p-2} X_{s} \, dK^c_s \xrightarrow[\varepsilon \rightarrow 0^+]{\mathbb{P}} \int_{0}^{t}e^{\frac{p}{2}\beta A_s} \left|X_s\right|^{p-1} \hat{X}_s \, dK^c_s.
		$$
		\item Similarly, as in \cite[Lemma 7]{Kruse18052016}, for the terms involving the first derivatives of $\nu_\varepsilon$ with respect to $\Delta K$, we have
		$$
		\sum_{0 < s \leq t} e^{\frac{p}{2}\beta A_s}\nu_\varepsilon(X_{s-})^{p-2} X_{s-} \Delta K_s \xrightarrow[\varepsilon \rightarrow 0^+]{\mathbb{P}} \sum_{0 < s \leq t} e^{\frac{p}{2}\beta A_s}\left|X_{s-}\right|^{p-1} \hat{X}_{s-} \Delta K_s.
		$$
	\end{itemize}
	Additionally, since the filtration $\mathbb{F}$ is quasi-left continuous, we deduce that the covariation process between $\int_0^\cdot \int_{\mathcal{U}} U_s(e) \tilde{\mu}(ds, de)$ and $K$ is zero. Using this, the remainder of the proof is identical to the proof of Lemma 7 in \cite{Kruse18052016}; thus, we omit the details here.
\end{proof}

A consequence of the previous lemma is the following Corollary in the one-dimensional case:
\begin{corollary}\label{cor}
	If $(Y,Z,U,K)$ is a $\mathbb{L}^p$-solution of RBSDEJs associated with parameters $(\xi,f,L)$, for $p\in(1,2)$, then
	\begin{eqnarray*}
		&&e^{\frac{p}{2}\beta A_t}|Y_t|^p+\frac{p}{2}\beta\int_t^Te^{\frac{p}{2}\beta A_s}|Y_s|^pdA_s+c(p)\int_t^Te^{\frac{p}{2}\beta A_s}|Y_s|^{p-2}|Z_s|^2\mathds{1}_{\{Y_s\neq0\}}ds\\
		&\leq&e^{\frac{p}{2}\beta A_T}|\xi|^p+p\int_t^Te^{\frac{p}{2}\beta A_s}|Y_s|^{p-1}\hat{Y}_sf(s,Y_s,Z_s,U_s)ds+p\int_t^Te^{\frac{p}{2}\beta A_s}|Y_{s-}|^{p-1}\hat{Y}_{s-}dK_s\\
		&&-p\int_t^Te^{\frac{p}{2}\beta A_s}|Y_s|^{p-1}\hat{Y}_sZ_sdB_s-p\int_t^T\int_{\mathcal{U}}e^{\frac{p}{2}\beta A_s}|Y_{s-}|^{p-1}\hat{Y}_{s-}U_s(e)\tilde{\mu}(ds,de)\\
		&&-\int_t^T\int_{\mathcal{U}}e^{\frac{p}{2}\beta A_s}\left[|Y_{s-}+U_s(e)|^{p}-|Y_{s-}|^p-p|Y_{s-}|^{p-1}\hat{Y}_{s-}U_s(e)\right]\mu(ds,de).
	\end{eqnarray*}
\end{corollary}

We also state the following Lemma, which can be found in \cite[Lemma 9]{Kruse18052016}.
\begin{lemma}\label{lem1}
	Let $p\in(1,2)$. Then we have
	\begin{eqnarray*}
		&&\hspace{-3cm}-\int_0^t\int_{\mathcal{U}}e^{\frac{p}{2}\beta A_s}\left[|Y_{s-}+U_s(e)|^{p}-|Y_{s-}|^p-p|Y_{s-}|^{p-1}\hat{Y}_{s-}U_s(e)\right]\mu(ds,de)\\
		&\leq&-c(p)\int_0^t\int_{\mathcal{U}}e^{\frac{p}{2}\beta A_s}|U_s(e)|^2\left(|Y_{s-}|^{2}\vee|Y_{s}|^2\right)^{\frac{p-2}{2}}\mathds{1}_{\{|Y_{s-}|\vee|Y_{s}|\neq 0\}}\mu(ds,de).
	\end{eqnarray*}
\end{lemma}
\begin{remark}\label{good rmq 1}
	Let $(Y,Z,V,K)$ be an $L^p$-solution of the RBSDEJ \eqref{basic equation}.\\
	Compared to \cite[Lemma 9]{Kruse18052016}, the state process $(Y_t)_{t \leq T}$ of the RBSDEJ \eqref{basic equation} has two type of jumps: Predictable ones that stems from the negative jumps of the lower barrier $L$, and  totally, inaccessible jumps, which stems from the Poisson random jump measure. Thus, for each  $t \in [0,T]$, we have $\int_{\mathcal{U}}U_t(e)\mu\left(\{t\},de\right)\Delta K_t=0$ and then, under the jump measure $\mu$, we can write 
	\begin{equation*}\label{Simple but import formula}
		\begin{split}
			&\int_0^t\int_{\mathcal{U}}e^{\frac{p}{2}\beta A_s}|U_s(e)|^2\left(|Y_{s-}|^{2}\vee|Y_{s-}+U_s(e)|^2\right)^{\frac{p-2}{2}}\mathds{1}_{\{|Y_{s-}|\vee|Y_{s-}+U_s(e)|\neq 0\}}\mu(ds,de)\\
			&= \int_0^t\int_{\mathcal{U}}e^{\frac{p}{2}\beta A_s}|U_s(e)|^2\left(|Y_{s-}|^{2}\vee|Y_{s}|^2\right)^{\frac{p-2}{2}}\mathds{1}_{\{|Y_{s-}|\vee|Y_{s}|\neq 0\}}\mu(ds,de).
		\end{split}
	\end{equation*}
	Consequently, in order to remain within the standard framework for stochastic integrals with respect to martingales driven by jump measures, where the integrands are predictable and integrable with respect to such martingales (see, for instance, Chapter II, Section 1d in \cite{jacod2013limit}), we retain the following formulation of the jump term when computing expectations. Although the two expressions below are mathematically equivalent, we choose to work with the form
	$$
	\int_0^\cdot \int_{\mathcal{U}} e^{\frac{p}{2} \beta A_s} |U_s(e)|^2 \left(|Y_{s-}|^2 \vee |Y_{s-} + U_s(e)|^2\right)^{\frac{p-2}{2}} \mathds{1}_{\{|Y_{s-}| \vee |Y_{s-} + U_s(e)| \neq 0\}} \, \mu(ds, de)
	$$
	instead of
	$$
	\int_0^\cdot \int_{\mathcal{U}} e^{\frac{p}{2} \beta A_s} |U_s(e)|^2 \left(|Y_{s-}|^2 \vee |Y_s|^2\right)^{\frac{p-2}{2}} \mathds{1}_{\{|Y_{s-}| \vee |Y_s| \neq 0\}} \, \mu(ds, de),
	$$
	in order to preserve the predictable structure of the integrands and ensure consistency with the general theory of stochastic integration for jump processes, in view of the issues observed in \cite{Kruse18052016} and corrected in \cite{Kruse17112017} for $p \in (1,2)$. Accordingly, we consider the associated predictable projection:
	$$
	\int_0^\cdot \int_{\mathcal{U}} e^{\frac{p}{2} \beta A_s} |U_s(e)|^2 \left(|Y_{s-}|^2 \vee |Y_{s-} + U_s(e)|^2\right)^{\frac{p-2}{2}} \mathds{1}_{\{|Y_{s-}| \vee |Y_{s-} + U_s(e)| \neq 0\}} \, \lambda(de) \, ds.
	$$
\end{remark}

Finally, we conclude this section with a basic inequality that, although simple, plays an important role in the next section, particularly in calculations involving exponents in the $p$-norm. For the reader's convenience, we also provide a short proof.

We state the following inequality:
\begin{equation}\label{basic inequality}
	\left(\sum_{i=1}^n |x_i| \right)^p \leq n^{p-1} \sum_{i=1}^n |x_i|^p, \qquad \forall (n,x,p) \in \mathbb{N}^\ast \times \mathbb{R}^n \times [1,+\infty).
\end{equation}
To see this, it suffices to consider the $\ell^p$ norm on $\mathbb{R}^n$, defined by
$$
\|x\|_1 = \sum_{i=1}^n |x_i|, \qquad \|x\|_p = \left( \sum_{i=1}^n |x_i|^p \right)^{\frac{1}{p}}, \qquad \forall x = (x_1, x_2, \dots, x_n) \in \mathbb{R}^n.
$$

It is well known that for all $p \geq 1$, the following estimate holds by Hölder's inequality:
$$
\|x\|_1 \leq n^{1 - \frac{1}{p}} \|x\|_p.
$$
Raising both sides to the power $p$, we obtain
$$
\left( \sum_{i=1}^n |x_i| \right)^p \leq n^{p - 1} \sum_{i=1}^n |x_i|^p,
$$
which yields to \eqref{basic inequality}.

\section{Existence and uniqueness of $\mathbb{L}^p$-solution to RBSDEJs via penalization method}\label{s4}
We begin by stating a preliminary result that will be used in the analysis developed in this section.
\begin{lemma}\label{rem1}
	Let $(Y,Z,U)\in \mathfrak{B}^{p}_\beta\times \mathcal{H}^{p}_\beta\times \mathfrak{L}^{p}_{\lambda,\beta}$, then for each $p\in(1,2)$
	$$\left(\int_0^te^{\frac{p}{2}\beta A_s}|Y_{s-}|^{p-1}\hat{Y}_{s-}\left(Z_sdB_s+\int_{\mathcal{U}}U_s(e)\tilde{\mu}(ds,de)\right)\right)_{t\leq T}
	\quad\mbox{is an uniformly integrable martingale.}$$
\end{lemma}	

\begin{proof}
	By using Young's inequality (i.e. $ab\leq\frac{p-1}{p}a^\frac{p}{p-1}+\frac{1}{p}b^p$) we have
	\begin{equation*}
		\begin{split}
			\mathbb{E}\left[ \left(\int_0^Te^{p\beta A_s}|Y_{s-}|^{2(p-1)}\|U_s\|_\lambda^2ds\right)^{\frac{1}{2}}\right] 
			&\leq\mathbb{E}\left[\sup_{0\leq t\leq T}e^{\frac{p-1}{2}\beta A_t}|Y_t|^{p-1}\left(\int_0^Te^{\beta A_s}\|U_s\|_\lambda^2ds\right)^{\frac{1}{2}}\right]\\
			&\leq\frac{p-1}{p}\mathbb{E}\left[\sup_{0\leq t\leq T}e^{\frac{p}{2}\beta A_t}|Y_t|^{p}\right]+\frac{1}{p}\mathbb{E}\left[ \left(\int_0^Te^{\beta A_s}\|U_s\|_\lambda^2ds\right)^{\frac{p}{2}}\right] \\
			&<+\infty,
		\end{split}
	\end{equation*}
	which is finite since $(Y,V)\in \mathfrak{B}^{p}_\beta\times \mathfrak{L}^{p}_{\lambda,\beta}$. In the same way
	\begin{equation*}
		\begin{split}
			\mathbb{E}\left[ \left(\int_0^Te^{p\beta A_s}|Y_{s-}|^{2(p-1)}|Z_s|^2ds\right)^{\frac{1}{2}}\right] 
			&\leq\frac{p-1}{p}\mathbb{E}\left[\sup_{0\leq t\leq T}e^{\frac{p}{2}\beta A_t}|Y_t|^{p}\right]+\frac{1}{p}\mathbb{E}\left[ \left(\int_0^Te^{\beta A_s}|Z_s|^2ds\right)^{\frac{p}{2}}\right]\\
			&<+\infty.
		\end{split}
	\end{equation*}
	The remainder of the proof is based on the application of Theorem 51 in \cite[p.~38]{Protter2005} and the Burkholder-Davis-Gundy's inequality (Theorem 48 in \cite[p. 193]{Protter2005}).
\end{proof}

Let us next recall, as noted in Remark \ref{Obse}, that one of the principal difficulties in deriving $\mathbb{L}^p$-estimates for $p \in (1,2)$ lies in the inapplicability of the classical Burkholder-Davis-Gundy inequality to the $\frac{p}{2}$-th moment of stochastic integrals with respect to $\tilde{\mu}$, since $\frac{p}{2} < 1$. For a detailed and related discussion, we refer the reader to Theorem VII.92, p.~304 in \cite{DellacherieMeyer1980}, as well as \cite{Lenglart1980}, \cite{MarinelliRoeckner2014}, and Theorem 48, p.~193 in \cite{Protter2005}. Furthermore, as stated in \cite{Kruse18052016} and corrected in \cite{Kruse17112017}, in the regime $p \in (1,2)$, when the generator $f$ depends explicitly on the variable $u$, the classical a priori estimates for $\mathbb{L}^p$-solutions $(Y, Z, U, K)$, as introduced in Definition \ref{Def}, cannot be derived directly using the stochastic Lipschitz condition $(\mathcal{H}2)$-(iii) on $f$ with respect to $u$ (see also Remark 14 in \cite{elmansouri2025}).

To address these challenges and prove the existence and uniqueness of an $\mathbb{L}^p$-solution to the RBSDEJ \eqref{basic equation} for $p \in (1,2)$, our approach is twofold. First, we analyze the case where the generator $f$ does not depend on the variable $u$, and derive the necessary a priori estimates as well as the existence and uniqueness of a solution. In the second step, we extend the result to the general case, where $f$ may depend on $u$, by employing a fixed point argument in an appropriately chosen Banach space.

\subsection{The case where the driver $f$ does not depend on the variables $(z,u)$}
In this initial step, we assume that the coefficient $f$ is independent of the variable $u$. For the sake of simplicity, we further suppose (in the existence result) that $f$ does not depend on $z$ either. However, it is important to emphasize that the method remains valid, even when $f$ depends on $z$ with only minor adjustments required. In particular, when $f$ depends on $z$, the main steps of the method still apply. The key point is that the stochastic integral with respect to the Brownian motion, associated with the $z$-variable, is a continuous (local) martingale. As such, standard estimates involving its predictable compensator can still be obtained via the Burkholder-Davis-Gundy  inequality (see Remark~\ref{Obse}). Furthermore, the terms involving the $z$-variable in the driver can be handled using classical arguments under assumption $(\mathcal{H}2)$-(iii), as we will see in the next section.

\subsubsection{A priori estimates and uniqueness}
In this part, we consider drivers $f_1$ and $f_2$ of the RBSDEJ \eqref{basic equation}, both of which are independent of the variable $u$ and depend only on $(y,z)$. Let $(Y^1, Z^1, U^1, K^1)$ and $(Y^2, Z^2, U^2, K^2)$ be two $\mathbb{L}^p$-solutions of the RBSDEJ \eqref{basic equation} associated with the data $(\xi_1, f_1, L)$ and $(\xi_2, f_2, L)$, respectively. More precisely, we define $f_i$ by setting $f_i(t, y, z) := f_i(t, y, z, 0)$ for all $(t, y, z,u) \in [0, T] \times \mathbb{R} \times \mathbb{R}^{d} \times \mathbb{L}^{2}_\lambda$ and for each $i \in \{1, 2\}$. We denote by $\bar{\mathcal{R}} := \mathcal{R}^1 - \mathcal{R}^2$ the difference between the corresponding components $\mathcal{R} \in \{Y, Z, U, \xi, f\}$. To clarify notation, we choose to work with $f(t,y,z)$ instead of $f(t,y,z,0)$

\begin{proposition}\label{Estimation p less stricly than 2}
	For any $\beta > \frac{2(p-1)}{p}$, there exists a constant $\mathfrak{C}_{p,\beta,\epsilon}$\footnote{Throughout this paper, $\mathfrak{c} > 0$ denotes a generic constant that may vary from line to line. In addition, the notation $\mathfrak{C}_{\gamma}$ will be used to emphasize the dependence of a constant $\mathfrak{C} > 0$ on a specific set of parameters $\gamma$, and may likewise vary from one line to another.} such that 
	\begin{equation*}
		\begin{split}
			&\mathbb{E}\left[\sup_{0 \leq t \leq T} e^{\frac{p}{2}\beta A_{t  }}\big|\bar{Y}_{t}\big|^p \right]+\mathbb{E}\left[ \int_{0}^{T}e^{\frac{p}{2}\beta A_s} \big|\bar{Y}_s\big|^p dA_s\right] +\mathbb{E}\left[\left(\int_{0}^{T}e^{\beta A_s} \big\|\bar{Z}_s\big\|^2 ds\right)^{\frac{p}{2}}\right]\\
			&+\mathbb{E}\left[\left(\int_{0}^{T}e^{\beta A_s} \big\|\bar{U}_s\big\|^2_{\mathbb{L}^2_\lambda} ds\right)^{\frac{p}{2}}\right]+\mathbb{E}\left[\left(\int_{0}^{T}e^{\beta A_s}\int_{\mathcal{U}} \big|\bar{U}_s(e)\big|^2 \mu(ds,de)\right)^{\frac{p}{2}}\right]\\
			&\leq \mathfrak{C}_{p,\beta,\epsilon} \left(\mathbb{E}\left[ e^{\frac{p}{2}\beta A_{T}}\big|\bar{\xi}\big|^p\right]  +\mathbb{E}\left[ \int_{0}^{T}e^{\beta A_s}\left|\bar{f}(s,Y^2_s,Z^2_s)\right|^pds\right] \right).
		\end{split}
	\end{equation*}
\end{proposition}

\begin{proof}
	Applying Corollary \ref{cor} in combination with Lemma \ref{lem1} and Remark \ref{good rmq 1} to the semimartingale $e^{\frac{p}{2}\beta A_{t}}\big|\bar{Y}_{t}\big|^p$, we obtain for any $t\in [0,T]$
	\begin{equation}\label{e01-u}
		\begin{split}
			&e^{\frac{p}{2}\beta A_t}|\bar Y_t|^p+\frac{p}{2}\beta\int_t^Te^{\frac{p}{2}\beta A_s}|\bar Y_s|^pdA_s+c(p)\int_t^Te^{\frac{p}{2}\beta A_s}|\bar Y_s|^{p-2}|\bar Z_s|^2\mathds{1}_{\{\bar{Y}_s\neq0\}}ds\\
			&\quad+c(p)\int_t^T\int_{\mathcal{U}}e^{\frac{p}{2}\beta A_s}|\bar U_s(e)|^2\left(|\bar Y_{s-}|^{2}\vee| \bar{Y}_{s-}+\bar{U}_s(e)|^2\right)^{\frac{p-2}{2}}\mathds{1}_{\{|\bar Y_{s-}|\vee|\bar Y_{s}+\bar{U}_s(e)|\neq 0\}}\mu(ds,de)\\
			&\leq e^{\frac{p}{2}\beta A_T}|\bar \xi|^p+p\int_t^Te^{\frac{p}{2}\beta A_s}|\bar Y_{s}|^{p-1}\hat{\bar Y}_{s}(f_1(s,Y^1_s,Z^1_s)-f_2(s,Y^2_s,Z^2_s))ds\\
			&\quad+p\int_t^Te^{\frac{p}{2}\beta A_s}|\bar Y_{s-}|^{p-1}\hat{\bar Y}_{s-}d\bar K_s-p\int_t^Te^{\frac{p}{2}\beta A_s}|\bar Y_s|^{p-1}\hat{\bar Y}_s\bar Z_sdB_s\\
			&\quad-p\int_t^T\int_{\mathcal{U}}e^{\frac{p}{2}\beta A_s}|\bar Y_{s-}|^{p-1}\hat{\bar Y}_{s-}\bar U_s(e)\tilde{\mu}(ds,de).
		\end{split}
	\end{equation}
	Using assumptions $(\mathcal{H}2)$-(ii)-(iii) on the generator $f$ along with Remark \ref{rmq essential} (for the the last line)
	\begin{equation}\label{generator p strictly less than 2}
		\begin{split}
			&p\big|\bar{Y}_s\big|^{p-1}\check{\widehat{Y}}_s \left(f_1(s,Y^1_s, Z^1_s)-f_2(s,Y^2_s, Z^2_s)\right)ds\\
			& = p\mathds{1}_{\{\bar{Y}_s \neq 0\}}\big|\widehat{Y}_s\big|^{p-2} \bar{Y}_t\left(f_1(s,Y^1_s, Z^1_s)-f_2(s,Y^2_s, Z^2_s)\right)ds\\
			& \leq p\mathds{1}_{\{\widehat{Y}_s \neq 0\}}\big|\bar{Y}_s\big|^{p-2}\left(\alpha_s \big|\widehat{Y}_s\big|^2 + \eta_s \big|\bar{Y}_s\big| \big|\bar{Z}_s\big|+\big|\bar{Y}_s\big|^{}\left|\widehat{f}(s,Y^2_s,Z^2_s)\right| \right)ds\\
			& \leq p\alpha_s \big|\widehat{Y}_s\big|^pds+p\big|\widehat{Y}_s\big|^{p-1}\left(\eta_s  \big|\bar{Z}_s\big|  +\left|\bar{f}(s,Y^2_s,Z^2_s)\right|\right)\mathds{1}_{\{\bar{Y}_s \neq 0\}}ds\\
			&\leq  p\left(\alpha_s+\frac{1}{(p-1)} a^2_s\right) \big|\bar{Y}_s\big|^pds +\frac{c(p)}{2}\big|\bar{Y}_s\big|^{p-2}\big|\bar{Z}_s\big|^2\mathds{1}_{\{\bar{Y}_s\neq 0\}} ds+p\big|\bar{Y}_s\big|^{p-1}\left|\bar{f}(s,Y^2_s,Z^2_s)\right|ds\\
			&\leq \frac{c(p)}{2}\big|\bar{Y}_s\big|^{p-2}\big|\bar{Z}_s\big|^2\mathds{1}_{\{\bar{Y}_s\neq 0\}} ds+p\big|\bar{Y}_s\big|^{p-1}\big|\bar{f}(s,Y^2_s,Z^2_s)\big|ds
		\end{split}
	\end{equation}
	Using Young's inequality and $(\mathcal{H}2)$-(v), we have
	\begin{equation}\label{tired 1}
		\begin{split}
			pe^{\frac{p}{2}\beta A_s} \big|\widehat{Y}_s\big|^{p-1}\big|\widehat{f}(s,Y^2_s,Z^2_s)\big|ds 
			&\leq(p-1)e^{\frac{p}{2}\beta A_s}\big|\widehat{Y}_s\big|^{p}dA_s+\frac{1}{\epsilon^{\frac{q}{2}}}e^{\frac{p}{2}\beta A_s}\big|\widehat{f}(s,Y^2_s,Z^2_s)\big|^pds\\
		\end{split}
	\end{equation}
	On the other hand, using \eqref{basic equation}-(ii) and the Skorokhod condition \eqref{basic equation}-(iii), we obtain
	$$
	dK^1_s = \mathds{1}_{\{Y^1_{s-} = L^1_{s-}\}} \, dK^1_s \quad \text{and} \quad dK^2_s = \mathds{1}_{\{Y^2_{s-} = L^2_{s-}\}} \, dK^2_s,
	$$
	which implies
	\begin{equation}\label{Skoro1}
		\begin{split}
			&\int_t^T e^{\frac{p}{2} \beta A_s} |\bar{Y}_{s-}|^{p-1} \hat{\bar{Y}}_{s-} \, d\bar{K}_s \\
			&= \int_t^T e^{\frac{p}{2} \beta A_s} |\bar{Y}_s|^{p-2} \mathds{1}_{\{\bar{Y} \neq 0\}} (Y^1_{s-} - L_{s-}) \, d\bar{K}_s 
			+ \int_t^T e^{\frac{p}{2} \beta A_s} |\bar{Y}_s|^{p-2} \mathds{1}_{\{\bar{Y} \neq 0\}} (L_{s-} - Y^2_{s-}) \, d\bar{K}_s \\
			&\leq - \int_t^T e^{\frac{p}{2} \beta A_s} |\bar{Y}_s|^{p-2} \mathds{1}_{\{\bar{Y} \neq 0\}} (Y^1_{s-} - L_{s-}) \, dK^2_s 
			+ \int_t^T e^{\frac{p}{2} \beta A_s} |\bar{Y}_s|^{p-2} \mathds{1}_{\{\bar{Y} \neq 0\}} (L_{s-} - Y^2_{s-}) \, dK^1_s \leq 0.
		\end{split}
	\end{equation}
	Returning to \eqref{e01-u}, and using \eqref{generator p strictly less than 2}, \eqref{tired 1}, and \eqref{Skoro1}, we take expectations and apply Lemma~\ref{rem1} with the choice $\beta > \frac{2(p-1)}{p}$. We then obtain the existence of a constant $\mathfrak{C}_{p,\beta,\epsilon}$ such that
	\begin{equation}\label{First inequality}
		\begin{split}
			&\mathbb{E} \left[ \int_{0}^{T} e^{\frac{p}{2} \beta A_s} |\bar{Y}_s|^p \, dA_s \right] 
			+ \mathbb{E} \left[ \int_0^T e^{\frac{p}{2} \beta A_s} |\bar{Y}_s|^{p-2} |\bar{Z}_s|^2 \mathds{1}_{\{\bar{Y}_s \neq 0\}} \, ds \right] \\
			&+ \mathbb{E} \left[ \int_0^T \int_{\mathcal{U}} e^{\frac{p}{2} \beta A_s} |\bar{U}_s(e)|^2 \left( |\bar{Y}_{s-}|^2 \vee |\bar{Y}_{s-} + \bar{U}_s(e)|^2 \right)^{\frac{p-2}{2}} \mathds{1}_{\{|\bar{Y}_{s-}| \vee |\bar{Y}_{s-} + \bar{U}_s(e)| \neq 0\}} \, \mu(ds,de) \right] \\
			&\leq \mathfrak{C}_{p,\beta,\epsilon} \left( \mathbb{E} \left[ e^{\frac{p}{2} \beta A_T} |\bar{\xi}|^p \right] 
			+ \mathbb{E} \left[ \int_0^T e^{\frac{p}{2} \beta A_s} |\bar{f}(s, Y^2_s, Z^2_s)|^p \, ds\right]  \right).
		\end{split}
	\end{equation}
	We now return once again to \eqref{e01-u}, and let $\{\tau_k\}_{k \geq 1}$ be the sequence of stopping times defined by
	$$
	\tau_k := \inf \left\{ t \geq 0 : \int_0^t \int_{\mathcal{U}} e^{\frac{p}{2} \beta A_s} |\bar U_s(e)|^2 \left( |\bar Y_{s-}|^2 \vee |\bar Y_{s-} + \bar U_s(e)|^2 \right)^{\frac{p-2}{2}} \mathds{1}_{\{|\bar Y_{s-}| \vee |\bar Y_{s-} + \bar U_s(e)| \neq 0\}} \, \lambda(de) \, ds \geq k \right\} \wedge T.
	$$
	This defines a localizing (of a fundamental) sequence for the local martingale
	$$
	\mathfrak{M} := \int_0^{\cdot} \int_{\mathcal{U}} e^{\frac{p}{2} \beta A_s} |\bar{U}_s(e)|^2 \left( |\bar{Y}_{s-}|^2 \vee |\bar{Y}_{s-} + \bar U_s(e)|^2 \right)^{\frac{p-2}{2}} \mathds{1}_{\{|\bar{Y}_{s-}| \vee |\bar{Y}_s| \neq 0\}} \, \tilde{\mu}(ds,de).
	$$
	Moreover, for each $k \geq 1$, we have
	\begin{equation}\label{Safe}
		\begin{split}
			&\mathbb{E} \left[ \int_0^{\tau_k} \int_{\mathcal{U}} e^{\frac{p}{2} \beta A_s} |\bar{U}_s(e)|^2 \left( |\bar{Y}_{s-}|^2 \vee |\bar{Y}_{s-} + \bar U_s(e)|^2 \right)^{\frac{p-2}{2}} \mathds{1}_{\{|\bar{Y}_{s-}| \vee |\bar{Y}_s| \neq 0\}} \, \mu(ds,de) \right] \\
			&= \mathbb{E} \left[ \int_0^{\tau_k} \int_{\mathcal{U}} e^{\frac{p}{2} \beta A_s} |\bar U_s(e)|^2 \left( |\bar Y_{s-}|^2 \vee |\bar Y_{s-} + \bar U_s(e)|^2 \right)^{\frac{p-2}{2}} \mathds{1}_{\{|\bar Y_{s-}| \vee |\bar Y_{s-} + \bar U_s(e)| \neq 0\}} \, \lambda(de) \, ds \right].
		\end{split}
	\end{equation}
	We note here that, in contrast to \cite[equality (31)]{Kruse18052016} and \cite[p.~2]{Kruse17112017}, by following Remark \ref{good rmq 1} and keeping the predictable process $U(\cdot)$ in the formulation of the compensator of $\mathfrak{M}$, which corresponds to the jump term at totally inaccessible jump times of $Y$ induced by the Poisson random measure $\mu$, the equality \eqref{Safe} remains valid within our framework.
	
	Next, by reapplying Corollary~\ref{cor} in combination with Lemma \ref{lem1} on the time interval $[0, \tau_k]$ for each $k \geq 1$, and by repeating the same computations that led to inequality \eqref{First inequality}, together with the validity of equality \eqref{Safe}, we obtain:
	\begin{equation}\label{First inequality.1}
		\begin{split}
			&\mathbb{E}\left[  \int_{0}^{\tau_k} e^{\frac{p}{2} \beta A_s} |\bar{Y}_s|^p \, dA_s \right] 
			+ \mathbb{E} \left[ \int_0^{\tau_k} e^{\frac{p}{2} \beta A_s} |\bar{Y}_s|^{p-2} |\bar{Z}_s|^2 \mathds{1}_{\{\bar{Y}_s \neq 0\}} \, ds \right] \\
			&+ \mathbb{E} \left[ \int_0^{\tau_k} \int_{\mathcal{U}} e^{\frac{p}{2} \beta A_s} |\bar U_s(e)|^2 \left( |Y_{s-}|^2 \vee |\bar Y_{s-} + \bar U_s(e)|^2 \right)^{\frac{p-2}{2}} \mathds{1}_{\{|\bar Y_{s-}| \vee |\bar Y_{s-} + \bar U_s(e)| \neq 0\}} \, \lambda(de) \, ds \right] \\
			&\leq \mathfrak{C}_{p,\beta,\epsilon} \left( \mathbb{E} \left[ e^{\frac{p}{2} \beta A_T} |\bar{Y}_{\tau_k}|^p \right] 
			+ \mathbb{E}\left[  \int_0^{\tau_k} e^{\frac{p}{2} \beta A_s} |\bar{f}(s, Y^2_s, Z^2_s)|^p \, ds\right]  \right).
		\end{split}
	\end{equation}
	By assumption $(\mathcal{H}1)$, we have $\mathbb{E}\left[e^{\frac{p}{2}\beta A_{T}} |\bar{\xi}|^p\right] < +\infty$. Then, by the dominated convergence theorem, it follows that 
	$$
	e^{\frac{p}{2} \beta A_{\tau_k}} |\bar{Y}_{\tau_k}|^p \longrightarrow e^{\frac{p}{2} \beta A_T} |\bar{\xi}|^p \quad \text{in } \mathbb{L}^1 \text{ as } k \rightarrow +\infty.
	$$
	Next, applying the monotone convergence theorem, and passing to the limit on both sides of inequality \eqref{First inequality.1} as $k \to +\infty$, we obtain
	\begin{equation}\label{First inequality.3}
		\begin{split}
			& \mathbb{E} \left[ \int_0^{T} \int_{\mathcal{U}} e^{\frac{p}{2} \beta A_s} |\bar U_s(e)|^2 \left( |\bar Y_{s-}|^2 \vee |\bar Y_{s-} + \bar U_s(e)|^2 \right)^{\frac{p-2}{2}} \mathds{1}_{\{|\bar Y_{s-}| \vee |\bar Y_{s-} + \bar U_s(e)| \neq 0\}} \, \lambda(de) \, ds \right] \\
			&\leq \mathfrak{C}_{p,\beta,\epsilon} \left( \mathbb{E} \left[ e^{\frac{p}{2} \beta A_T} |\bar{\xi}|^p \right] 
			+ \mathbb{E}\left[  \int_0^{T} e^{\frac{p}{2} \beta A_s} |\bar{f}(s, Y^2_s, Z^2_s)|^p \, ds\right]  \right).
		\end{split}
	\end{equation}
	
	Let us now deal with the uniform estimation for the state variable $Y$. From \eqref{e01-u} applying \eqref{Skoro1} and taking the supremum, we get
	\begin{equation}\label{sup for p less than 2}
		\begin{split}
			\mathbb{E}\left[\sup_{0 \leq t \leq T } e^{\frac{p}{2}\beta A_{t  }}\big|\bar{Y}_{t}\big|^p \right]
			\leq& \mathbb{E}\left[ e^{\frac{p}{2}\beta A_{T}}\big|\bar{\xi}\big|^p\right] +p\mathbb{E}\int_{0}^{T}e^{\frac{p}{2}\beta A_s} \big|\bar{Y}_s\big|^{p-1}\left|\bar{f}(s,Y^2_s,Z^2_s)\right|ds  \\
			&+p\mathbb{E}\left[\sup_{t \in [0,T]}\left|\int_{t}^{T}e^{\frac{p}{2}\beta A_s} \big|\bar{Y}_s\big|^{p-1} \hat{\bar{Y}}_s \widehat{Z}_s dB_s\right|\right]\\
			& +p\mathbb{E}\left[\sup_{t \in [0,T]}\left|\int_{t}^{T}e^{\frac{p}{2}\beta A_s}\int_{E} \big|\bar{Y}_{s-}\big|^{p-1} \hat{\bar{Y}}_{s-} \bar{U}_s(e)\tilde{\mu}(ds,de)\right|\right]
		\end{split}
	\end{equation}
	Using Lemma 8 in \cite{Kruse18052016} and the BDG inequality, we have
	\begin{equation*}
		\begin{split}
			&p\mathbb{E}\left[\sup_{0 \leq t \leq T }\left|\int_{t}^{T}e^{\frac{p}{2}\beta A_s} \big|\bar{Y}_s\big|^{p-1} \hat{\bar{Y}}_s \bar{Z}_s dB_s\right|\right]\\
			&\leq p \mathfrak{c} \mathbb{E}\left[\left(\int_{0}^{T}e^{p\beta A_s} \big|\bar{Y}_s\big|^{2(p-1)} \big|\bar{Z}_s\big|^2 \mathds{1}_{\{\bar{Y}_s \neq 0\}} ds\right)^{\frac{1}{2}}\right]\\
			&\leq \mathbb{E}\left[\left( \sup_{t \in [0,T]}e^{\frac{p}{4}\beta}\big|\bar{Y}_s\big|^{\frac{p}{2}}\right) \left( p^2 \mathfrak{c}^2\int_{0}^{T}e^{\frac{p}{2}\beta A_s} \big|\bar{Y}_s\big|^{p-2} \big|\bar{Z}_s\big|^2 \mathds{1}_{\{\bar{Y}_s \neq 0\}} ds\right)^{\frac{1}{2}}\right]\\
			&\leq \frac{1}{6} \mathbb{E}\left[\sup_{t \in [0,T]}e^{\frac{p}{2}\beta}\big|\bar{Y}_s\big|^{p}\right]+\frac{3}{2}p^2 \mathfrak{c}^2\mathbb{E}\left[\int_{0}^{T}e^{\frac{p}{2}\beta A_s} \big|\bar{Y}_s\big|^{p-2} \big|\bar{Z}_s\big|^2 \mathds{1}_{\{\bar{Y}_s \neq 0\}} ds\right]
		\end{split}
	\end{equation*}
	the last line is due to the basic inequality $ab \leq \frac{1}{6}a^2+\frac{3}{2} b^2$ (recall that $\mathfrak{c}>0$ is the BDG universal constant). 
	
	Again, by the BDG inequality, we have
	\begin{equation*}
		\begin{split}
			&\mathbb{E}\left[ \sup_{0\leq t\leq T}\left|\int_t^T\int_{\mathcal{U}}e^{\frac{p}{2}\beta A_s}|\widehat{Y}_{s-}|^{p-1}\hat{\bar{Y}}_{s-} \bar{U}_s(e)\tilde{\mu}(ds,de)\right|\right] \\
			&\leq \mathfrak{c}\mathbb{E}\left[\left(\int_0^T\int_{\mathcal{U}}e^{p\beta A_s}|\bar{Y}_{s-}|^{2(p-1)}|\bar{U}_s(e)|^2 \mu (ds,de)\right)^{\frac{1}{2}}\right]
		\end{split}
	\end{equation*}
	Following this, Remark \ref{good rmq 1} and using the basic inequality $ab \leq \frac{1}{6}a^2+\frac{3}{2} b^2$, we have
	\begin{equation*}
		\begin{split}
			&p\mathfrak{c}\mathbb{E}\left[\left(\int_0^T\int_{\mathcal{U}}e^{p\beta A_s}|\bar{Y}_{s-}|^{2(p-1)}|\bar{U}_s(e)|^2 \mu (ds,de)\right)^{\frac{1}{2}}\right]\\
			& \leq p\mathfrak{c}\mathbb{E}\left[\left(\int_0^T\int_{\mathcal{U}}e^{p\beta A_s}\big(|\bar{Y}_{s-}|^{2} \vee |\bar{Y}_{s-}+U_s(e)|^{2}\big)^{p-1}|\bar{U}_s(e)|^2 \mu (ds,de)\right)^{\frac{1}{2}}\right]\\
			&\leq\mathfrak{c}\mathbb{E}\left[\left( \sup_{0\leq t\leq T}e^{\frac{p}{4}\beta A_t}|\bar{Y}_{t}|^{\frac{p}{2}}\right) \right.\\
			&\left. \qquad \times p \left( \int_0^T\int_{\mathcal{U}}e^{\frac{p}{2}\beta A_s}\left(|\bar{Y}_{s-}|^{2}\vee|\bar{Y}_{s-}+\bar U_s(e)|^2\right)^{\frac{p-2}{2}}\mathds{1}_{\{|\widehat{Y}_{s-}|\vee|\bar{Y}_{s}|\neq 0\}}|\bar{U}_s(e)|^2 \mu(ds,de)\right)^{\frac{1}{2}}\right]\\
			&\leq\frac{1}{6}\mathbb{E}\left[\sup_{0\leq t\leq T}e^{\frac{p}{2}\beta A_t}|\bar{Y}_{t}|^{p}\right]\\
			&\qquad+\frac{3}{2}p^2 \mathfrak{c}^2\mathbb{E}\left[\int_0^T\int_{E}e^{\frac{p}{2}\beta A_s}\left(|\bar{Y}_{s-}|^{2}\vee|\bar{Y}_{s-}+\bar U_s(e)|^2\right)^{\frac{p-2}{2}}\mathds{1}_{\{|\bar{Y}_{s-}|\vee\bar{Y}_{s-}+\bar U_s(e)\neq 0\}}|\bar{U}_s(e)|^2 \mu(ds,de)\right].
		\end{split}
	\end{equation*}
	Finally using Young and Jensen inequalities, we get 
	\begin{equation}\label{Paranoia}
		\begin{split}
			&p\mathbb{E}\left[\int_{0}^{T}e^{\frac{p}{2}\beta A_s} \big|\bar{Y}_s\big|^{p-1}\left|\bar{f}(s,Y^2_s,Z^2_s)\right|ds\right]\\
			&=p	\mathbb{E}\left[\int_{0}^{T}e^{\frac{p-1}{2}\beta A_s} \big|\bar{Y}_s\big|^{p-1} e^{\frac{\beta}{2} A_s}\left|\bar{f}(s,Y^2_s,Z^2_s)\right|ds\right]\\
			& \leq p\mathbb{E}\left[\left(\frac{1}{6(p-1)}\right)^{\frac{p-1}{p}}\left(\sup_{0 \leq t \leq T}e^{\frac{p-1}{2}\beta A_s} \big|\bar{Y}_s\big|^{p-1}\right)\left(\frac{1}{6(p-1)}\right)^{\frac{1-p}{p}}\int_{0}^{T} e^{\frac{\beta}{2} A_s}\left|\bar{f}(s,Y^2_s,Z^2_s)\right|ds\right]\\
			& \leq \frac{1}{6}\mathbb{E}\left[\sup_{0 \leq t \leq T}e^{\frac{p}{2}\beta A_s}\big|\bar{Y}_t\big|^{p}\right]+\left(6 T(p-1)\right)^{p-1}\mathbb{E}\left[\int_{0}^{T} e^{\frac{p}{2}\beta A_s}\left|\bar{f}(s,Y^2_s,Z^2_s)\right|^p ds\right]
		\end{split}
	\end{equation}
	Using \eqref{First inequality} and \eqref{sup for p less than 2}, we get
	\begin{equation}\label{First inequality.4}
		\begin{split}
			&\mathbb{E} \left[ \sup_{0 \leq t \leq T } e^{\frac{p}{2}\beta A_{t  }}\big|\bar{Y}_{t}\big|^p \right]  
			\leq \mathfrak{C}_{p,\beta,\epsilon} \left( \mathbb{E} \left[ e^{\frac{p}{2} \beta A_T} |\bar{\xi}|^p \right] 
			+ \mathbb{E} \left[ \int_0^T e^{\frac{p}{2} \beta A_s} |\bar{f}(s, Y^2_s, Z^2_s)|^p \, ds\right]  \right).
		\end{split}
	\end{equation}
	
	We now deal with the desire estimations for the terms:
	$$
	\mathbb{E}\left[\left(\int_{0}^{T}\int_{\mathcal{U}}e^{\beta A_s} \big|\bar{U}_s(e)\big|^2 \mu(ds,de)\right)^{\frac{p}{2}}\right] \quad \text{ and } \quad \mathbb{E}\left[\left(\int_{0}^{T}e^{\beta A_s} \big\|\bar{U}_s\big\|^2_{\mathbb{L}^2_\lambda} ds\right)^{\frac{p}{2}}\right].
	$$
	Note that for the above term involving the stochastic integral with respect to the jump measure ${\mu}(dt,de)$ ot it's predictable compensator $\lambda(de)dt$, we need to use an approximation approach. This regularization is the one used, for instance, in \cite[p. 20]{elmansouri2025} or \cite[pp. 32-33]{Kruse18052016}.\\ 
	Let $\varepsilon > 0$ and define the function $\nu_\varepsilon : \mathbb{R} \rightarrow \mathbb{R}$ by
	$$
	\nu_\varepsilon(y) := \left(|y|^2 + \varepsilon^2\right)^{\frac{1}{2}}, \quad \text{for all } y \in \mathbb{R}.
	$$
	from the definition of the function $\nu_\varepsilon$, it is easy to see that 
	\begin{equation}
		\begin{split}
			e^{\frac{2-p}{4}\beta A_s}\left( \nu_\varepsilon\left( \big|\widehat{Y}_{s-}\big|\vee \big|\widehat{Y}_{s-}+\bar{U}_s(e)\big|\right) \right) ^{\frac{2-p}{2}}&=\left( e^{\frac{1}{2}\beta A_s}\nu_\varepsilon\left( \big|\bar{Y}_{s-}\big|\vee \big|\bar{Y}_{s-}+\bar{U}_s(e)\big|\right) \right) ^{\frac{2-p}{2}}\\
			&=\left[ \nu_{\varepsilon_p}\left( e^{\frac{1}{2}\beta A_s}\left( \big|\bar{Y}_{s-}\big|\vee \big|\bar{Y}_{s-}+\bar{U}_s(e)\big|\right) \right) \right] ^{\frac{2-p}{2}}	
		\end{split}
	\end{equation}
	with $\varepsilon_p:=\varepsilon e^{\frac{1}{2}\beta A_s}$ which tend to zero as $\varepsilon \downarrow 0$ since $e^{\frac{1}{2}\beta A_s} \leq e^{\beta A_T}<+\infty$ a.s. This last property is derived from assumption $(\mathcal{H}2)$-(iv) implying $\mathbb{E}\left[\int_{0}^{T}e^{\beta A_s} ds\right]<+\infty$. Additionally, for any $p \in (1,2)$, note that $\lim\limits_{\varepsilon \downarrow 0} \left(\nu_\varepsilon(y)\right)^{p}=\big|y\big|^p$ and $\lim\limits_{\varepsilon \downarrow 0} \left(\nu_\varepsilon(y)\right)^{p-2}=\big|y\big|^{p-2}\mathds{1}_{\{y \neq 0\}}$. 
	\begin{equation}\label{R1}
		\begin{split}
			&\mathbb{E}\left[\left(\int_{0}^{T}\int_{\mathcal{U}}e^{\beta A_s} \big|\bar{U}_s(e)\big|^2 \mu(ds,de)\right)^{\frac{p}{2}}\right]\\
			&=\mathbb{E}\left[\left(\int_{0}^{T}e^{\frac{2-p}{2}\beta A_s}\int_{\mathcal{U}}\left(\nu_\varepsilon(\big|\bar{Y}_{s-}\big|\vee \big|\bar{Y}_{s-}+\bar{U}_s(e)\big|\big)\right)^{2-p} \right.\right.\\
			&\left.\left.\qquad\qquad \times \left(\nu_\varepsilon(\big|\bar{Y}_{s-}\big|\vee \big|\bar{Y}_{s-}+\bar{U}_s(e)\big|\big)\right)^{p-2}e^{\frac{p}{2} \beta A_s}  \big|\bar{U}_s(e)\big|^2 \mu(ds,de)\right)^{\frac{p}{2}}\right]\\
			&=\mathbb{E}\left[\left(\int_{0}^{T}\left(e^{\frac{2-p}{4}\beta A_s}\left(\nu_\varepsilon(\big|\widehat{Y}_{s-}\big|\vee \big|\bar{Y}_s\big|\big)\right)^{\frac{2-p}{2}}\right)^2 \right.\right. \\
			&\left.\left. \qquad\qquad \times \int_{\mathcal{U}}  \left(\nu_\varepsilon(\big|\widehat{Y}_{s-}\big|\vee \big|\bar{Y}_{s-}+\bar{U}_s(e)\big|\big)\right)^{p-2} e^{\frac{p}{2} \beta A_s}\big|\bar{U}_s(e)\big|^2 \mu(ds,de)\right)^{\frac{p}{2}}\right]\\
			&=\mathbb{E}\left[\left(\int_{0}^{T}\left(\nu_{\varepsilon_p}(e^{\frac{1}{2}\beta A_s}\big|\bar{Y}_{s-}\big|\vee \big|\bar{Y}_{s}\big|)\right)^{2-p} \right.\right.\\
			&\left.\left. \qquad\qquad \times \int_{\mathcal{U}}  \left(\nu_\varepsilon(|\bar{Y}_{s-}| \vee |\bar{Y}_{s-}+\bar{U}_s(e)|)\right)^{p-2}e^{\frac{p}{2} \beta A_s}  \big|\bar{U}_s(e)\big|^2 \mu(ds,de)\right)^{\frac{p}{2}}\right]\\
			&\leq \mathbb{E}\left[\left(\nu_{\varepsilon_p}\left(\left( e^{\frac{1}{2}\beta A} \bar{Y}\right)_\ast\right)\right)^{p\frac{2-p}{2}}\left(\int_{0}^{T} e^{\frac{p}{2}\beta A_s}\int_{\mathcal{U}} \left(\nu_\varepsilon\big(\big|\bar{Y}_{s-}\big|\vee \big|\bar{Y}_{s-}+\bar{U}_s(e)\big|\big)\right)^{p-2}  \big|\widehat{U}_s(e)\big|^2 \mu(ds,de)\right)^{\frac{p}{2}}\right]\\
			& \leq \frac{2-p}{2}\mathbb{E}\left[\left(\nu_{\varepsilon_p}\left(\left( e^{\frac{1}{2}\beta A} \widehat{Y}\right)_\ast\right)\right)^{p}\right]+\frac{p}{2} \mathbb{E}\left[\int_{0}^{T} e^{\frac{p}{2}\beta A_s}\int_{\mathcal{U}} \left(\nu_\varepsilon\big(\big|\widehat{Y}_{s-}\big|\vee \big|\bar{Y}_{s-}+\bar{U}_s(e)\big|\big)\right)^{p-2}  \big|\bar{U}_s(e)\big|^2 \mu(ds,de)\right].
		\end{split}
	\end{equation}
	Let $\varepsilon$ tend to zero.  We can use a convergence theorem, which is a consequence of the estimations \eqref{First inequality} and \eqref{First inequality.4} to derive
	\begin{equation*}
		\begin{split}
			&\mathbb{E}\left[\left(\int_{0}^{T} \int_{\mathcal{U}} e^{\beta A_s}\big|\bar{U}_s(e)\big|^2 \mu(ds,de)\right)^{\frac{p}{2}}\right]\\
			& \leq \frac{2-p}{2}\mathbb{E}\left[\sup_{0 \leq t \leq T} e^{\frac{p}{2}\beta A_{t  }}\big|\bar{Y}_{t}\big|^p\right]\\
			&\qquad+\frac{p}{2} \mathbb{E}\left[\int_{0}^{T}e^{\frac{p}{2}\beta A_s}\int_{\mathcal{U}}\left(|\bar{Y}_{s-}|^{2}\vee|\bar{Y}_{s-}+\bar{U}_s(e)|^2\right)^{\frac{p-2}{2}}\mathds{1}_{\{|\widehat{Y}_{s-}|\vee|\bar{Y}_{s-}+\bar{U}_s(e)|\neq 0\}}\big|\bar{U}_s(e)\big|^2\mu(ds,de)\right]\\
			& \leq  \mathfrak{C}_{p,\beta,\epsilon}  \left(\mathbb{E}\left[ e^{\frac{p}{2}\beta A_{T}}\big|\bar{\xi}\big|^p\right]  +\mathbb{E}\left[ \int_{0}^{T}e^{\beta A_s}\left|\bar{f}(s,Y^2_s,Z^2_s)\right|^p ds\right] \right).
		\end{split}
	\end{equation*}
	By a similar argument and using \eqref{First inequality.3}, we can derive that 
	\begin{equation}\label{R3}
		\begin{split}
			&\mathbb{E}\left[\left(\int_{0}^{T}e^{\beta A_s} \big\|\bar{U}_s\big\|^2_{\mathbb{L}^2_\lambda} ds\right)^{\frac{p}{2}}\right]\\
			& \leq \frac{2-p}{2}\mathbb{E}\left[\sup_{0 \leq t \leq T} e^{\frac{p}{2}\beta A_{t  }}\big|\bar {Y}_{t}\big|^p\right]\\
			&\qquad+\frac{p}{2} \mathbb{E} \left[ \int_0^{T}  e^{\frac{p}{2} \beta A_s} \int_{\mathcal{U}}  \left( |\bar Y_{s-}|^2 \vee |\bar Y_{s-} + \bar U_s(e)|^2 \right)^{\frac{p-2}{2}} \mathds{1}_{\{|\bar Y_{s-}| \vee |\bar Y_{s-} + \bar U_s(e)| \neq 0\}} |\bar U_s(e)|^2 \, \lambda(de) \, ds \right]\\
			& \leq  \mathfrak{C}_{p,\beta,\epsilon}\left(\mathbb{E}\left[ e^{\frac{p}{2}\beta A_{T}}\big|\bar{\xi}\big|^p\right]  +\mathbb{E}\left[ \int_{0}^{T}e^{\beta A_s}\left|\bar{f}(s,Y^2_s,Z^2_s)\right|^p ds\right] \right).
		\end{split}
	\end{equation}
	
	To conclude the proof of the proposition, it remains to show the estimation for the remaining term 
	$$
	\mathbb{E}\left[\left(\int_{0}^{T}e^{\beta A_s} \big\|\widehat{Z}_s\big\|^2 ds\right)^{\frac{p}{2}}\right].
	$$
	To this end, note that if $\widehat{Y}=0$ then after writing the BSDE \eqref{basic equation}-(i) forwardly, we derive that the continuous martingale part $\left(\int_{0}^{t}\bar{Z}_s dB_s\right)_{t \leq T}$ have finite variation over $[0,T]$. Since it is predictable, we derive from \cite[Ch I. Corollary 3.16]{jacod2013limit} that $\int_{0}^{\cdot}\bar{Z}_s dB_s=0$ up to an evanescent set. Then, using this with Young's inequality and \eqref{First inequality.1}, \eqref{First inequality.4},  we have
	\begin{equation}\label{R2}
		\begin{split}
			&\mathbb{E}\left[\left(\int_{0}^{T}e^{\beta A_s} \big|\bar{Z}_s\big|^2 ds\right)^{\frac{p}{2}}\right]\\
			&=\mathbb{E}\left[\left(\int_{0}^{T}e^{\beta A_s} \big|\bar{Z}_s\big|^2 \mathds{1}_{\{\bar{Y}_s\neq0\}} ds\right)^{\frac{p}{2}}\right]\\
			& \leq \mathbb{E}\left[\left( \sup_{t \in [0,T]}\left(e^{\frac{p(2-p)}{4}\beta A_s} \big|\bar{Y}_s\big|^{\frac{p(2-p)}{2}}\right)\right) \left(\int_{0}^{T}  \big|\bar{Y}_s\big|^{p-2} \big|\bar{Z}_s\big|^2\mathds{1}_{\{\bar{Y}_s\neq0\}} ds\right)^{\frac{p}{2}}\right]\\
			& \leq \frac{2-p}{2}\mathbb{E}\left[\sup_{t \in [0,T]} e^{\frac{p}{2}\beta A_{t  }}\big|\bar{Y}_{t}\big|^p\right]+\frac{p}{2} \mathbb{E}\left[\int_{0}^{T}e^{\frac{p}{2}\beta A_s}\big|\bar{Y}_{s}\big|^{p-2} \big|\bar{Z}_s\big|^2 \mathds{1}_{\{\bar{Y}_{s}\neq 0\} }ds\right]\\
			& \leq   \mathfrak{C}_{p,\beta,\epsilon} \left(\mathbb{E}\left[ e^{\frac{p}{2}\beta A_{T}}\big|\bar{\xi}\big|^p\right]  +\mathbb{E}\left[ \int_{0}^{T}e^{\beta A_s}\left|\bar{f}(s,Y^2_s,Z^2_s)\right|^pds\right] \right).
		\end{split}
	\end{equation}
	Completing the proof.
\end{proof}

Following the arguments developed in Proposition~\ref{Estimation p less stricly than 2}, we obtain the following estimate for $\mathbb{L}^p$-solutions of the RBSDEJ \eqref{basic equation}, in the case where the driver is independent of the $u$-variable.
\begin{proposition}\label{Estimation p less stricly than 2.1}
	Given any $\mathbb{L}^p$-solution $(Y, Z, U, K)$ of the RBSDEJ \eqref{basic equation} associated with the data $(\xi, f, L)$, under assumptions $(\mathcal{H}1)$--$(\mathcal{H}3)$ and for a sufficiently large parameter $\beta$\footnote{The choice of the parameter $\beta$ here differs from that in Proposition~\ref{Estimation p less stricly than 2}, as additional Lebesgue-Stieltjes terms of the form $\int_{t}^{T} e^{\frac{p}{2} \beta A_s} |Y_s|^p \, dA_s$ appear in the estimates, as seen for example in inequalities \eqref{tired 1} and \eqref{K.3}.}, assume further that $f(t, y, z, u) = f(t, y, z, 0)$ for all $(t, y, z, u) \in [0,T] \times \mathbb{R}^{1+d} \times \mathbb{L}^2_\lambda$. Then, the following estimate holds:
	\begin{equation*}
		\begin{split}
			&\mathbb{E}\left[\sup_{0 \leq t \leq T} e^{\frac{p}{2}\beta A_{t  }}\big|{Y}_{t}\big|^p \right]+\mathbb{E}\left[ \int_{0}^{T}e^{\frac{p}{2}\beta A_s} \big|{Y}_s\big|^p dA_s\right] +\mathbb{E}\left[\left(\int_{0}^{T}e^{\beta A_s} \big|{Z}_s\big|^2 ds\right)^{\frac{p}{2}}\right]\\
			&+\mathbb{E}\left[\left(\int_{0}^{T}e^{\beta A_s} \big\|{U}_s\big\|^2_{\mathbb{L}^2_\lambda} ds\right)^{\frac{p}{2}}\right]+\mathbb{E}\left[\left(\int_{0}^{T}e^{\beta A_s}\int_{\mathcal{U}} \big|{U}_s(e)\big|^2 \mu(ds,de)\right)^{\frac{p}{2}}\right]+\mathbb{E}\left[\left|K_T\right|^p\right]\\
			&\leq \mathfrak{C}_{p,\beta,\epsilon} \left(\mathbb{E}\left[ e^{\frac{p}{2}\beta A_{T}}\big|{\xi}\big|^p\right]  +\mathbb{E}\left[ \int_{0}^{T}e^{\beta A_s}\left|\varphi_s\right|^pds\right]+\mathbb{E}\left[\sup\limits_{0\leq t\leq T}\left|e^{\frac{q}{2}\beta A_t}L_t^+\right|^{p}\right] \right).
		\end{split}
	\end{equation*}
\end{proposition}
\begin{proof}
	As previously mentioned, it suffices to follow a similar argument to that used in the proof of Proposition~\ref{Estimation p less stricly than 2}. The corresponding generalized Itô's formula, resulting from the application of Corollary~\ref{cor}, Lemma~\ref{lem1}, and Remark~\ref{good rmq 1} to the semimartingale $e^{\frac{p}{2} \beta A_t} |Y_t|^p$, is given by:
	\begin{equation}\label{e01-u.1}
		\begin{split}
			&e^{\frac{p}{2}\beta A_t}|\bar Y_t|^p+\frac{p}{2}\beta\int_t^Te^{\frac{p}{2}\beta A_s}| Y_s|^pdA_s+c(p)\int_t^Te^{\frac{p}{2}\beta A_s}| Y_s|^{p-2}| Z_s|^2\mathds{1}_{\{{Y}_s\neq0\}}ds\\
			&\quad+c(p)\int_t^T\int_{\mathcal{U}}e^{\frac{p}{2}\beta A_s}| U_s(e)|^2\left(| Y_{s-}|^{2}\vee| {Y}_{s-}+{U}_s(e)|^2\right)^{\frac{p-2}{2}}\mathds{1}_{\{| Y_{s-}|\vee| Y_{s}+{U}_s(e)|\neq 0\}}\mu(ds,de)\\
			&\leq e^{\frac{p}{2}\beta A_T}| \xi|^p+p\int_t^Te^{\frac{p}{2}\beta A_s}| Y_{s}|^{p-1}\hat{ Y}_{s}f(s,Y_s,Z_s)ds+p\int_t^Te^{\frac{p}{2}\beta A_s}| Y_{s-}|^{p-1}\hat{ Y}_{s-}d K_s\\
			&\quad-p\int_t^Te^{\frac{p}{2}\beta A_s}| Y_s|^{p-1}\hat{ Y}_s Z_sdB_s
			-p\int_t^T\int_{\mathcal{U}}e^{\frac{p}{2}\beta A_s}| Y_{s-}|^{p-1}\hat{ Y}_{s-} U_s(e)\tilde{\mu}(ds,de).
		\end{split}
	\end{equation}
	For the estimation related to the driver $f$, the argument follows the same lines as in \eqref{generator p strictly less than 2} and \eqref{tired 1}. However, the main difference lies in the use of the Skorokhod condition \eqref{basic equation}-(iii), as emphasized in \eqref{Skoro1}. In this context, we employ the following estimate, valid for any $\varrho > 0$:
	\begin{equation}\label{Skoro1.1}
		\begin{split}
			\int_t^T e^{\frac{p}{2} \beta A_s} |{Y}_{s-}|^{p-1} \hat{{Y}}_{s-} \, d{K}_s 
			&= \int_t^T e^{\frac{p}{2} \beta A_s} |{L}_{s-}|^{p-1} \hat{{L}}_{s-} \, d{K}_s \\
			&\leq \big( \sup_{0 \leq t \leq T} e^{\frac{p}{2} \beta A_t}({L}^+_{t})^{p-1}\big)  \left(K_T-K_t\right)\\
			&\leq \left(\frac{\varrho}{p}\right)^{\frac{q}{p}}\frac{p-1}{p}\big( \sup_{0 \leq t \leq T} e^{\frac{q p}{2} \beta A_t}({L}^+_{t})^{p}\big)  +\frac{1}{ \varrho}\left(K_T-K_t\right)^p.
		\end{split}
	\end{equation}
	To estimate the term $\left(K_T - K_t\right)^p$, we return to the BSDE \eqref{basic equation}-(i), which yields
	$$
	K_T-K_t=\xi-\int_{t}^{T}f(s,Y_s,Z_s)ds+\int_{t}^{T}Z_s dB_s+\int_{t}^{T}\int_{\mathcal{U}} U_s(e)\tilde{\mu}(ds,de),\quad t \in [0,T].
	$$
	Using the basic inequality \eqref{basic inequality}, we obtain
	\begin{equation}\label{Inequation for K}
		\left(K_T-K_t\right)^p \leq 4^{p-1}\left(\left|\xi\right|^p+\left(\int_{t}^{T}|f(s,Y_s,Z_s)|ds\right)^p+\left|\int_{t}^{T}Z_s dB_s\right|^p+\left|\int_{t}^{T}\int_{\mathcal{U}} U_s(e)\tilde{\mu}(ds,de)\right|^p\right).
	\end{equation}
	For the last two martingale terms, since $p > 1$, we can apply the BDG inequality to obtain
	\begin{equation}\label{K.1}
		\begin{split}
			\mathbb{E}\left[\sup_{0 \leq t \leq T } \left|\int_{t}^{T}Z_s dB_s\right|^p\right] \leq \mathfrak{c} \mathbb{E}\left[\left(\int_{0}^{T}|Z_s|^2ds\right)^{\frac{p}{2}}\right]\leq \mathfrak{c} \mathbb{E}\left[\left(\int_{0}^{T}e^{\beta A_s}|Z_s|^2ds\right)^{\frac{p}{2}}\right]
		\end{split}
	\end{equation}
	and similarly, we have
	\begin{equation}\label{K.2}
		\begin{split}
			\mathbb{E}\left[\sup_{0 \leq t \leq T } \left|\int_{t}^{T}\int_{\mathcal{U}} U_s(e)\tilde{\mu}(ds,de)\right|^p\right] \leq \mathfrak{c} \mathbb{E}\left[\left(\int_{0}^{T}e^{\beta A_s}\int_{\mathcal{U}}|U_s(e)|^2 {\mu}(ds,de)\right)^{\frac{p}{2}}\right].
		\end{split}
	\end{equation}
	For the driver term, we apply assumptions $(\mathcal{H}2)$-(iii)-(iv), together with Hölder's and Jensen's inequalities, which yields
	\begin{equation}\label{K.3}
		\begin{split}
			&\left( \int_{t}^{T}\left| f(s,Y_s,Z_s) \right|ds \right)^p\\
			& \leq \left( \int_{t}^{T}\left(\left| f(s,Y_s,0) \right|+\left| f(s,Y_s,Z_s)-f(s,Y_s,0) \right|\right)ds \right)^p \\
			&\leq \left(\int_{t}^{T}\left\{\varphi_s + \phi_s |Y_s| ds+\eta_s|Z_s|\right\} ds\right)^p\\
			& \leq 3^{p-1} \left(\left(\int_{t}^{T}\varphi_s ds \right)^p+\left(\int_{t}^{T}|Y_s| dA_s\right)^p+\left(\int_{t}^{T}|Z_s| ds\right)^p\right)\\
			& \leq 3^{p-1} \left(T^{p-1}\int_{t}^{T}e^{\beta A_s}\left|\varphi_s\right|^p ds +\left(\int_{t}^{T}e^{-\frac{q}{2}\beta A_s}dA_s\right)^{p-1}\left(\int_{t}^{T}e^{\frac{p}{2}\beta A_s}|Y_s|^p dA_s\right) \right.\\
			&\qquad\qquad \left.+\left(\int_{t}^{T}e^{-\frac{q}{2} \beta A_s}\eta^q_s ds\right)^{p-1}\left(\int_{t}^{T}e^{\frac{p}{2}\beta A_s} |Z_s|^p ds\right)\right) \\
			& \leq 3^{p-1}\left(T^{p-1}\int_{t}^{T}e^{\beta A_s}\left|\varphi_s\right|^p ds+\left(\frac{2(p-1)}{p}\right)^{p-1}\int_{t}^{T}e^{\frac{p}{2}\beta A_s}|Y_s|^p dA_s \right.\\
			&\qquad\qquad\left.+\left(\frac{2(p-1)}{p}\right)^{p-1}\left(\int_{t}^{T}e^{\frac{p}{2}\beta A_s} |Z_s|^p ds\right)\right)\\
			& \leq 3^{p-1}\left(T^{p-1}\int_{t}^{T}e^{\beta A_s}\left|\varphi_s\right|^p ds+\left(\frac{2(p-1)}{p}\right)^{p-1}\int_{t}^{T}e^{\frac{p}{2}\beta A_s}|Y_s|^p dA_s\right.\\
			&\qquad\qquad\quad \left.+\left(T-t\right)^{\frac{2-p}{2}}\left( \frac{2(p-1)}{p}\right)^{p-1}\left(\int_{t}^{T}e^{\beta A_s} |Z_s|^2 ds\right)^{\frac{p}{2}}\right).
		\end{split}
	\end{equation}
	Now, returning to \eqref{e01-u.1}, and using Lemma~\ref{rem1} together with similar estimates as in \eqref{First inequality}, \eqref{First inequality.1}, and \eqref{Skoro1.1}, we obtain
	\begin{equation}\label{First inequality.5}
		\begin{split}
			&\mathbb{E} \left[ \int_{0}^{T} e^{\frac{p}{2} \beta A_s} |\bar{Y}_s|^p \, dA_s \right] 
			+ \mathbb{E} \left[ \int_0^T e^{\frac{p}{2} \beta A_s} |\bar{Y}_s|^{p-2} |\bar{Z}_s|^2 \mathds{1}_{\{\bar{Y}_s \neq 0\}} \, ds \right] \\
			&\quad+ \mathbb{E} \left[ \int_0^T \int_{\mathcal{U}} e^{\frac{p}{2} \beta A_s} |\bar{U}_s(e)|^2 \left( |\bar{Y}_{s-}|^2 \vee |\bar{Y}_{s-} + \bar{U}_s(e)|^2 \right)^{\frac{p-2}{2}} \mathds{1}_{\{|\bar{Y}_{s-}| \vee |\bar{Y}_{s-} + \bar{U}_s(e)| \neq 0\}} \, \mu(ds,de) \right] \\
			&\quad+ \mathbb{E} \left[ \int_0^{\tau_k} \int_{\mathcal{U}} e^{\frac{p}{2} \beta A_s} |\bar U_s(e)|^2 \left( |Y_{s-}|^2 \vee |\bar Y_{s-} + \bar U_s(e)|^2 \right)^{\frac{p-2}{2}} \mathds{1}_{\{|\bar Y_{s-}| \vee |\bar Y_{s-} + \bar U_s(e)| \neq 0\}} \, \lambda(de) \, ds \right] \\
			&\leq \mathfrak{C}_{p,\beta,\epsilon} \left( \mathbb{E} \left[ e^{\frac{p}{2} \beta A_T} |{\xi}|^p \right] 
			+ \mathbb{E} \left[ \int_0^T e^{\frac{p}{2} \beta A_s} |\varphi_s|^p \, ds\right]\right)\\
			&\quad+\mathfrak{C}_{p,\beta,\epsilon}\left(\frac{\varrho}{p}\right)^{\frac{q}{p}}\frac{p-1}{p}\mathbb{E}\left[\sup\limits_{0\leq t\leq T}\left|e^{\frac{q}{2}\beta A_t}L_t^+\right|^{p}\right]+\frac{\mathfrak{C}_{p,\beta,\epsilon}}{\varrho}\mathbb{E}\left[\left|K_T\right|^p\right].
		\end{split}
	\end{equation}
	Following the application of the BDG inequality in \eqref{First inequality.4} together with \eqref{First inequality.5}, we also deduce that the term $\|Y\|^p_{\mathcal{S}^{p}_\beta}$ is bounded by the right-hand side of \eqref{First inequality.5}, up to a suitably adjusted constant $\mathfrak{C}_{p,\beta,\epsilon}$.\\ 
	Using this result, along with a similar estimate for the control variable $Z$ established in \eqref{R2}, and applying the regularization procedure used in \eqref{R1} and \eqref{R3} for the control variable $U$, we derive that
	$$
	\mathbb{E}\left[\left(\int_{0}^{T}e^{\beta A_s} \big|{Z}_s\big|^2 ds\right)^{\frac{p}{2}}\right]+\mathbb{E}\left[\left(\int_{0}^{T} \int_{\mathcal{U}} e^{\beta A_s}\big|{U}_s(e)\big|^2 \mu(ds,de)\right)^{\frac{p}{2}}\right]+\mathbb{E}\left[\left(\int_{0}^{T}e^{\beta A_s} \big\|\bar{U}_s\big\|^2_{\mathbb{L}^2_\lambda} ds\right)^{\frac{p}{2}}\right]
	$$
	is also controlled by the right-hand side of \eqref{First inequality.5}. It then remains to take the expectation on both sides of \eqref{Inequation for K}, and to use the previously obtained estimates \eqref{K.1}, \eqref{K.2}, and \eqref{K.3}, together with the argumentation described above for the triplet $(Y, Z, U)$. \\
	Finally, by choosing $\varrho$ sufficiently large such that $\varrho > \mathfrak{C}_{p,\beta,\epsilon}$, where $\mathfrak{C}_{p,\beta,\epsilon}$ denotes the constant arising from the above considerations, we arrive at the following estimate:
	\begin{equation*}
		\begin{split}
			&\mathbb{E}\left[\sup_{0 \leq t \leq T} e^{\frac{p}{2}\beta A_{t  }}\big|{Y}_{t}\big|^p \right]+\mathbb{E}\left[ \int_{0}^{T}e^{\frac{p}{2}\beta A_s} \big|{Y}_s\big|^p dA_s\right] +\mathbb{E}\left[\left(\int_{0}^{T}e^{\beta A_s} \big|{Z}_s\big|^2 ds\right)^{\frac{p}{2}}\right]\\
			&+\mathbb{E}\left[\left(\int_{0}^{T}e^{\beta A_s} \big\|{U}_s\big\|^2_{\mathbb{L}^2_\lambda} ds\right)^{\frac{p}{2}}\right]+\mathbb{E}\left[\left(\int_{0}^{T}e^{\beta A_s}\int_{\mathcal{U}} \big|{U}_s(e)\big|^2 \mu(ds,de)\right)^{\frac{p}{2}}\right]\\
			&\leq \mathfrak{C}_{p,\beta,\epsilon} \left(\mathbb{E}\left[ e^{\frac{p}{2}\beta A_{T}}\big|{\xi}\big|^p\right]  +\mathbb{E}\left[ \int_{0}^{T}e^{\beta A_s}\left|\varphi_s\right|^pds\right]+\mathbb{E}\left[\sup\limits_{0\leq t\leq T}\left|e^{\frac{q}{2}\beta A_t}L_t^+\right|^{p}\right] \right).
		\end{split}
	\end{equation*}
	By applying the above estimate to \eqref{Inequation for K}, we obtain the result stated in Proposition~\ref{Estimation p less stricly than 2.1}, which completes the proof.
\end{proof}

Another result that follows from Proposition~\ref{Estimation p less stricly than 2} is the following uniqueness corollary:
\begin{corollary}
	Let $(\xi, f, L)$ be a set of data satisfying assumptions $(\mathcal{H}1)$--$(\mathcal{H}3)$ for a sufficiently large $\beta$, and suppose that the driver $f$ is independent of the variable $u$. Then, the RBSDEJ \eqref{basic equation} associated with $(\xi, f, L)$ admits at most one $\mathbb{L}^p$-solution for any $p \in (1,2)$.
\end{corollary}
\subsubsection{Existence}
We now aim to establish the existence and uniqueness of an $\mathbb{L}^p$-solution to the RBSDEJ \eqref{basic equation} associated with the data $(\xi, f, L)$, under the simplifying assumption that the driver $f$ is independent of the variables $(z,u)$. Our approach is inspired by the penalization method developed in the Brownian setting by \cite[Section 4]{HamadenePopier2012}, which we extend to the jump framework by incorporating additional technical challenges arising from discontinuities, stochastic monotonicity in the coefficients, and RCLL obstacles with general jumps, going beyond the classical assumptions of Lipschitz continuity and continuous barriers.

To begin, we define $\mathfrak{f}(t,y) := f(t,y,z,u)$, where the driver $\mathfrak{f}$ satisfies assumptions $(\mathcal{H}2)$-(i)-(ii) and $(\mathcal{H}2)$-(iv)-(vi). For clarity, we group these four conditions under the new label $(\mathcal{H}2')$, with the following correspondence: $(\mathcal{H}2')$-(i)-(ii) refer to $(\mathcal{H}2)$-(i)-(ii), $(\mathcal{H}2')$-(iii) refers to $(\mathcal{H}2)$-(iv), and $(\mathcal{H}2')$-(iv)-(v) correspond to $(\mathcal{H}2)$-(v)-(vi).

The main result of the current section is giving in the following Lemma:
\begin{theorem}\label{lem2}
	Assume that $(\mathcal{H}1)$, $(\mathcal{H}2')$ and $(\mathcal{H}3)$ hold. Then, the RBSDEJs associated with parameters $(\xi,\mathfrak{f},L)$ has a unique $\mathbb{L}^p$-solution for each $p\in(1,2)$.
\end{theorem}
\begin{proof}
	Let $\mathfrak{f}_n(t, y) := \mathfrak{f}(t, y) + n(y - L_t)^-$. Thanks to \cite[Theorem 16]{elmansouri2025}\footnote{For $p \in (1,2)$ and its conjugate $q = \frac{p}{p-1}$, we have $pq = \frac{p^2}{p-1} > 2$, since $p^2 - 2p + 2 > 0$ for all $p > 1$. Thus, $\mathbb{E}\big[ \int_{0}^{T} e^{\beta A_s} L_s^+ \, ds \big] \leq T \, \mathbb{E} \big[ \sup_{0 \leq t \leq T} \big| e^{\frac{q}{2} \beta A_t} L_t^+ \big|^p \big] < +\infty$. This, together with assumption $(\mathcal{H}2')$-(iii), implies that $\mathbb{E}\big[ \int_{0}^{T} e^{\beta A_s} |\mathfrak{f}_n(s, 0)|^p \, ds \big] < +\infty$.}, for each $n \in \mathbb{N}$, there exists a process $(Y^n, Z^n, U^n) \in \mathfrak{B}^{p}_\beta \times \mathcal{H}^{p}_\beta \times \mathfrak{L}^{p}_{\lambda,\beta}$ that solves the following BSDE with jumps:
	\begin{equation}\label{penalized}
		Y^n_{t}=\xi +\int_{t}^{T} \mathfrak{f}(s,Y^n_s)ds+n\int_{t}^{T}(Y^n_s-L_s)^-ds-\int_{t}^{T}Z^n_{s}dB_{s}
		-\int_t^T\int_{\mathcal{U}}U_s^n(e)\tilde{\mu}(ds,de),\quad t \in [0,T].
	\end{equation}
	Let $K^n_t:=n\int_{0}^{t}(Y^n_s-L_s)^-ds$. The proof will be divided into four steps.
	
	\textbf{Step 1:} There exists a positive constant $\mathfrak{C}_{p,\beta,\varepsilon}$ that is independent on $n$ such that
	\begin{equation*}
		\begin{split}
			&\|Y^n\|_{\mathfrak{B}^{p}_\beta}^p+|Z^n\|_{\mathcal{H}^{p}_\beta}^p+\|U^n\|_{\mathfrak{L}^{p}_{\lambda,\beta}}^p+\|U^n\|^p_{\mathfrak{L}^{p}_{\mu,\beta}}+\mathbb{E}\left[ \left|K^n_{T} \right|^p\right]\nonumber\\
			&\leq\mathfrak{C}_{p,\beta,\epsilon}\left( \mathbb{E}\left[e^{\frac{p}{2}\beta A_T}|\xi|^p\right] +\mathbb{E}\left[\int_{0}^{T}e^{\beta A_s}\left|\varphi_s\right|^pds\right] +\mathbb{E}\bigg[\sup_{0\leq t\leq T}\big|e^{\frac{q}{2}\beta A_t}L_t^+\big|^{p}\bigg]\right) .
		\end{split}
	\end{equation*}

	\begin{proof}
		By applying Corollary \ref{cor} combining with Lemma \ref{lem1}, we get
		\begin{eqnarray}\label{e01-n}
			&&e^{\frac{p}{2}\beta A_t}|Y^n_t|^p+\frac{p}{2}\beta\int_t^Te^{\frac{p}{2}\beta A_s}|Y^n_s|^pdA_s+c(p)\int_t^Te^{\frac{p}{2}\beta A_s}|Y^n_s|^{p-2}|Z^n_s|^2\mathds{1}_{\{Y^n_s\neq0\}}ds\nonumber\\
			&&+c(p)\int_t^T\int_{\mathcal{U}}e^{\frac{p}{2}\beta A_s}|V^n_s(e)|^2\left(|Y^n_{s-}|^{2}\vee|Y^n_{s-}+U_s(e)|^2\right)^{\frac{p-2}{2}}\mathds{1}_{\{|Y^n_{s-}|\vee|Y^n_{s-}+U_s(e)|\neq 0\}}\mu(ds,de)\nonumber\\
			&\leq&e^{\frac{p}{2}\beta A_T}|\xi|^p+p\int_t^Te^{\frac{p}{2}\beta A_s}|Y^n_s|^{p-1}\hat{Y}^n_s\mathfrak{f}(s,Y^n_s)ds+p\int_t^Te^{\frac{p}{2}\beta A_s}|Y^n_{s-}|^{p-1}\hat{Y}^n_{s-}dK^n_s\nonumber\\
			&&-p\int_t^Te^{\frac{p}{2}\beta A_s}|Y^n_s|^{p-1}\hat{Y}^n_sZ^n_sdB_s-p\int_t^T\int_{\mathcal{U}}e^{\frac{p}{2}\beta A_s}|Y^n_{s-}|^{p-1}\hat{Y}^n_{s-}U^n_s(e)\tilde{\mu}(ds,de).
		\end{eqnarray}
		Using assumption $(\mathcal{H}2')$-(ii) and Remark \ref{rmq essential}, we get
		$$
		|Y^n_s|^{p-1}\hat{Y}^n_s\mathfrak{f}(s,Y^n_s) \leq |Y^n_s|^{p-1}\mathfrak{f}(s,0).
		$$
		Then by applying Young's inequality (i.e. $ab\leq\frac{p-1}{p}a^\frac{p}{p-1}+\frac{1}{p}b^p$), the H\"{o}lder's inequality (i.e. $\int|hg|\leq \left(\int|h|^\frac{2}{2-p}\right)^\frac{2-p}{2}\left(\int|g|^\frac{2}{p}\right)^\frac{p}{2}$) and condition $(\mathcal{H}2')$-(iv), we have 
		\begin{eqnarray*}
			p\int_t^Te^{\frac{p}{2}\beta A_s}|Y^n_s|^{p-1}\mathfrak{f}(s,0)ds
			&=&p\int_t^T\left(e^{\frac{p-1}{2}\beta A_s}|\zeta_s|^{\frac{2(p-1)}{p}}|Y^n_s|^{p-1}\right)\left(e^{\frac{\beta}{2}A_s}|\zeta_s|^{\frac{2-p}{p}}\left|\frac{\mathfrak{f}(s,0)}{\zeta_s}\right|\right)ds\\
			&\leq&(p-1)\int_t^Te^{\frac{p}{2}\beta A_s}|Y^n_s|^{p}dA_s+\int_t^Te^{\frac{p}{2}\beta A_s}|\zeta_s|^{2-p}\left|\frac{\mathfrak{f}(s,0)}{\zeta_s}\right|^pds\\
			&\leq&(p-1)\int_t^Te^{\frac{p}{2}\beta A_s}|Y^n_s|^{p}dA_s+\frac{1}{\epsilon^{\frac{q(p-1)}{2}}}\int_t^Te^{\frac{p}{2}\beta A_s}\left|\varphi_s\right|^p ds\nonumber\\
		\end{eqnarray*}	
		Moreover, as the function $x \in \mathbb{R} \mapsto |x|^{p-1} \hat{x}$ is non-decreasing for any $p > 1$, it follows that
		$$
		|Y^n_{s-}|^{p-1} \hat{Y}^n_s \leq |L^+_s|^{p-1}, \quad \text{whenever } Y^n_s < L_s,
		$$
		and then
		\begin{equation*}
			\int_t^Te^{\frac{p}{2}\beta A_s}|Y^n_{s-}|^{p-1}\hat{Y}^n_s(Y^n_s-L_s)^-ds\leq\int_t^Te^{\frac{p}{2}\beta A_s}|L^+_s|^{p-1}(Y^n_s-L_s)^-ds.
		\end{equation*}
		Hence, for each $\varrho>0$
		\begin{eqnarray*}\label{eq9-n}
			\int_t^Te^{\frac{p}{2}\beta A_s}|Y^n_{s-}|^{p-1}\hat{Y}^n_{s-}dK^n_s
			&\leq&\left(\int_0^T\left(e^{\frac{p}{2}\beta A_s}|L_s^+|^{p-1}\right)^\frac{p}{p-1}dK^n_s\right)^\frac{p-1}{p}.\left|K^n_T\right|^\frac{1}{p}\nonumber\\
			&\leq&\left(\int_0^T\left|e^{\frac{q}{2}\beta A_s}L_s^+\right|^{p}dK^n_s\right)^\frac{p-1}{p}.\left|K^n_T\right|^\frac{1}{p}\nonumber\\
			&\leq&\left(\sup_{0\leq t\leq T}\left|e^{\frac{q}{2}\beta A_t}L_t^+\right|^{p}\right)^\frac{p-1}{p}.\left|K^n_T\right|\nonumber\\
			&\leq&\left(\frac{\varrho}{p}\right)^ {\frac{q}{p}} \frac{p-1}{p}\left(\sup_{0\leq t\leq T}\big|e^{\frac{q}{2}\beta A_t}L_t^+\big|^{p}\right)+\frac{1}{\varrho}\left|K^n_T\right|^p,
		\end{eqnarray*}
		where we have used again H\"{o}lder's and Young's inequalities. \\
		The rest of the proof is derived by following similar arguments as the one employed in the proof of Propositions \ref{Estimation p less stricly than 2} and \ref{Estimation p less stricly than 2.1}.
	\end{proof}
	
	\textbf{Step 2:} There exists an RCLL process $Y$ such that $Y\geq L$ and $$
	\mathbb{E}\left[\sup\limits_{0\leq t\leq T}e^{\frac{p}{2}\beta A_t}|(Y^n_t-L_t)^-|^p\right]\xrightarrow[n\to +\infty]{}0.
	$$
	
	\begin{proof}
		For each $n\in\mathbb{N}$ we recall that
		\begin{equation}\label{e000-n}
			Y^n_{t}=\xi +\int_{t}^{T}\mathfrak{f}(s,Y^n_s)ds+K^n_T-K^n_t-\int_{t}^{T}Z^n_{s}dB_{s}-\int_t^T\int_{\mathcal{U}}V_s^n(e)\tilde{\mu}(ds,de).
		\end{equation}
		First, by the comparison theorem, we have $Y^n \leq Y^{n+1}$ for all $n \in \mathbb{N}$. In this case, the comparison result holds in the standard sense (see, e.g., \cite[Section 5]{Kruse18052016}), since the driver of the BSDE \eqref{e000-n} does not depend on the $u$-component. Therefore, there exists an $\mathbb{F}$-optional process $Y$ such that, $\mathbb{P}$-a.s., for every $t \in [0,T]$, we have $Y_t := \lim_{n \rightarrow +\infty} Y^n_t$. Using this convergence, together with the continuity of the mapping $\mathfrak{f}$ with respect to the $y$-variable (assumption $(\mathcal{H}2')$-(v)) and Remark~\ref{rmq essential}, we obtain that $\mathfrak{f}(t, Y^n_t) \searrow \mathfrak{f}(t, Y_t)$, $\mathbb{P}$-a.s., for every $t \in [0,T]$. By applying the monotone convergence theorem, we deduce that $\lim_{n \rightarrow +\infty} \int_0^t \mathfrak{f}^+(s, Y^n_s)\, ds = \int_0^t \mathfrak{f}^+(s, Y_s)\, ds$, and similarly, $\lim_{n \rightarrow +\infty} \int_0^t \mathfrak{f}^-(s, Y^n_s)\, ds = \int_0^t \mathfrak{f}^-(s, Y_s)\, ds$, so that $\lim_{n \rightarrow +\infty} \int_0^t \mathfrak{f}(s, Y^n_s)\, ds = \int_0^t \mathfrak{f}(s, Y_s)\, ds$, $\mathbb{P}$-a.s., for all $t \in [0,T]$. Alternatively, by exploiting the convergence $\mathfrak{f}(t, Y^n_t) \rightarrow \mathfrak{f}(t, Y_t)$, the linear growth of $\mathfrak{f}$ given by assumption $(\mathcal{H}2')$-(iii), the basic inequality \eqref{basic inequality}, the uniform estimate for the sequence $\{Y^n\}_{n \geq 0}$ established in \textbf{Step 1}, and the Lebesgue dominated convergence theorem, one may also obtain this convergence for the sequence $\{\int_{0}^{t}\mathfrak{f}(s, Y^n_s) ds\}_{n \geq 0}$ uniformity in $t$ in $\mathbb{L}^p$.
		
		By Fatou's lemma and \textbf{Step 1}, we have
		\begin{equation*}
			\begin{split}
				&\mathbb{E}\left[\sup_{0\leq t\leq T}e^{\frac{p}{2}\beta A_t}|Y_{t}|^{p}\right]+\mathbb{E}\left[\int_0^Te^{\frac{p}{2}\beta A_s}|Y^n_s|^pdA_s\right]\\
				&\leq\liminf_{n\rightarrow +\infty}\left\{\mathbb{E}\left[\sup_{0\leq t\leq T}e^{\frac{p}{2}\beta A_t}|Y_{t}^{n}|^{p}\right]+\mathbb{E}\left[\int_0^Te^{\frac{p}{2}\beta A_s}|Y^n_s|^pdA_s\right]\right\}<+\infty .
			\end{split}
		\end{equation*}
		Then $Y \in \mathfrak{B}^{p}_\beta$.  Next, from (\ref{e000-n}), we have
		\begin{equation*}
			\mathbb{E}[Y^n_{0}]=\mathbb{E}\left[\xi +\int_{0}^{T}\mathfrak{f}(s,Y_s)ds\right]+ \mathbb{E}\left[n\int_{0}^{T}(Y^n_{s}-L_{s})^-ds\right].
		\end{equation*}
		By passing to the limit as $n \to +\infty$, we obtain
		$
		\mathbb{E}\left[\int_0^T (Y_s - L_s)^- \, ds\right] = 0,
		$ which implies that $(Y_t - L_t)^- = 0$ for all $t \in [0, T)$, and hence $Y_t \geq L_t$ on $[0,T)$, since the process $Y - L$ is RCLL. Moreover, by assumption $(\mathcal{H}3)$-(i), we have $L_T \leq \xi$, so it follows that $Y_t \geq L_t$ for all $t \in [0,T]$, $\mathbb{P}$-a.s. In addition, we observe that $e^{\frac{p}{2}\beta A_t}(Y^n_t - L_t)^- \searrow 0$ for all $t \in [0,T]$, $\mathbb{P}$-a.s. This convergence also extends to the left limits, due to the fact that the jumps of $Y^n$ and $L$ occur at totally inaccessible times driven by the Poisson random measure $\mu$. Thus, from $\mathbb{E}\left[\int_0^T (Y_s - L_s)^- \, ds\right] = 0$, we also have $(Y_{t-} - L_{t-})^- = 0$ for all $t \in [0,T]$ (see also \cite[Proposition 4.1]{hamadene2016reflected}). Consequently, by applying a generalized version of Dini's lemma (see page 202 in \cite{DellacherieMeyer1980}), we deduce that $
		\sup_{0 \leq t \leq T} e^{\frac{p}{2}\beta A_t}(Y^n_t - L_t)^- \searrow 0 \quad \text{a.s.}$ Finally, since $|(Y^n_t - L_t)^-|^p \leq 2^{p-1}(|Y^0_t|^p + |L^+_t|^p)$ by inequality \eqref{basic inequality}, the uniform estimate from \textbf{Step 1}, together with the Lebesgue dominated convergence theorem, implies that
		$$
		\mathbb{E}\left[\sup_{0\leq t\leq T}e^{\frac{p}{2}\beta A_t}|(Y^n_t-L_t)^-|^p\right]\xrightarrow[n\to +\infty]{}0\quad \mbox{a.s}.
		$$
	\end{proof}
	
	\textbf{Step 3:} There exists an $\mathbb{F}$-adapted process $(Z,U,K)$ such that with the limiting process $Y$ satisfies the following convergence result:
	\begin{eqnarray*}
		&&\|Y^{n}-Y\|_{\mathcal{S}^p}^{p}+\|Z^{n}-Z\|_{\mathcal{H}^p}^{p}+\|U^n-U\|_{\mathfrak{L}^p_{\lambda}}^p
		+\|K^{n}-K\|_{\mathcal{S}^p}^{p}\xrightarrow[n\to +\infty]{}0.
	\end{eqnarray*}	
	
	\begin{proof}
		Let us put $\Re^{n,m}=\Re^n-\Re^m$ for each $n\geq m\geq0$ and for $\Re\in\{Y,Z,U,K\}$. The Corollary \ref{cor} combining with Lemma \ref{lem1} implies
		\begin{equation}\label{eq6}
			\begin{split}
				&|Y^{n,m}_t|^p+c(p)\int_t^Te^{\frac{p}{2}\beta A_s}|Y^{n,m}_s|^{p-2}|Z^{n,m}_s|^2\mathds{1}_{\{Y^n_s\neq Y^m_s\}}ds\\
				&+c(p)\int_t^T\int_{\mathcal{U}}|U^{n,m}_s(e)|^2\left(|Y^{n,m}_{s-}|^{2}\vee|Y^{n,m}_{s-}+U^{n,m}_s(e)|^2\right)^{\frac{p-2}{2}}\mathds{1}_{\{|Y^{n,m}_{s-}|\vee|Y^{n,m}_{s-}+U^{n,m}_s(e)|\neq 0\}}\mu(ds,de)\\
				&\leq p\int_t^T|Y^{n,m}_s|^{p-1}\hat{Y}^{n,m}_s\left(\mathfrak{f}(s,Y^n_s)-\mathfrak{f}(s,Y^m_s)\right)ds+p\int_t^T|Y^{n,m}_{s-}|^{p-1}\hat{Y}^{n,m}_{s-}dK^{n,m}_s\\
				&-p\int_t^T|Y^{n,m}_{s}|^{p-1}\hat{Y}^{n,m}Z^{n,m}_{s}dB_s
				-p\int_t^T\int_{\mathcal{U}}|Y^{n,m}_{s-}|^{p-1}\hat{Y}^{n,m}_{s-}U^{n,m}_{s}(e)\tilde{\mu}(ds,de).
			\end{split}
		\end{equation}
		By using Remark \ref{rmq essential} for the generator $\mathfrak{f}$ with assumption $(\mathcal{H}2')$-(ii), Lemma \ref{rem1} for the martingale part in \eqref{eq6} and taking expectation, we get
		\begin{equation}\label{CV}
			\begin{split}
				&\mathbb{E}\left[ |Y^{n,m}_t|^p\right] +c(p)\mathbb{E}\left[ \int_t^Te^{\frac{p}{2}\beta A_s}|Y^{n,m}_s|^{p-2}|Z^{n,m}_s|^2\mathds{1}_{\{Y^n_s\neq Y^m_s\}}ds\right] \\
				&+c(p)\mathbb{E}\left[ \int_t^T\int_{\mathcal{U}}|U^{n,m}_s(e)|^2\left(|Y^{n,m}_{s-}|^{2}\vee|Y^{n,m}_{s-}+U^{n,m}_s(e)|^2\right)^{\frac{p-2}{2}}\mathds{1}_{\{|Y^{n,m}_{s-}|\vee|Y^{n,m}_{s-}+U^{n,m}_s(e)|\neq 0\}}\mu(ds,de)\right] \\
				&\leq  p\mathbb{E}\left[ \int_t^T|Y^{n,m}_{s-}|^{p-1}dK^{n,m}_s\right] \\
				&\leq p\left(\mathbb{E}\left[ \sup_{0\leq t\leq T}|(Y^m_t-L_t)^-|^{p}\right] \right)^\frac{p-1}{p}\left(\mathbb{E}\left[ |K^n_T|^p\right] \right)^\frac{1}{p}+p\left(\mathbb{E}\left[ \sup_{0\leq t\leq T}|(Y^n_t-L_t)^-|^{p}\right] \right)^\frac{p-1}{p}\left(\mathbb{E}\left[ |K^m_T|^p\right] \right)^\frac{1}{p}\\
				&\xrightarrow[n,m\to +\infty]{}0.
			\end{split}
		\end{equation}
		On the other hand, by Burkholder-Davis-Gundy's inequality and following similar computations as in the proof of Proposition \ref{Estimation p less stricly than 2}, we can easily derive the existence of two non-negative constants $c$ and $\rho$ such that
		\begin{equation*}
			\begin{split}
				&\mathbb{E}\left[ \sup_{0\leq t\leq T}\left|\int_0^t|Y^{n,m}_s|^{p-1}\hat{Y}^{n,m}_sZ^{n,m}_sdB_s\right|\right] \\
				&\leq c\mathbb{E}\left[ \left(\sup_{0\leq t\leq T}|Y^{n,m}_{t}|^{p}\int_0^Te^{\frac{p}{2}\beta A_s}|Y^{n,m}_s|^{p-2}\mathds{1}_{\{Y^n_{s}\neq Y^m_{s}\}}|Z^{n,m}_s|^2ds\right)^{\frac{1}{2}}\right] \\
				&\leq\frac{1}{\rho}\mathbb{E}\left[\sup_{0\leq t\leq T}|Y^{n,m}_{t}|^{p}\right]+\rho c^2\mathbb{E}\left[ \int_0^T|Y^{n,m}_s|^{p-2}\mathds{1}_{\{Y^n_{s}\neq Y^m_{s}\}}|Z^{n,m}_s|^2ds\right] 
			\end{split}
		\end{equation*}
		and similarly, we found
		\begin{equation*}
			\begin{split}
				&\mathbb{E}\left[ \sup_{0\leq t\leq T}\left|\int_0^t\int_{\mathcal{U}}|Y^{n,m}_{s-}|^{p-1}\hat{Y}^{n,m}_sU^{n,m}_s(e)\tilde{\mu}(ds,de)\right|\right] \\
				&\leq\frac{1}{\rho}\mathbb{E}\left[\sup_{0\leq t\leq T}e^{\frac{p}{2}\beta A_t}|Y^{n,m}_{t}|^{p}\right]\\
				&\quad+\rho c^2\mathbb{E}\left[ \int_0^T\int_{\mathcal{U}}\left(|Y^{n,m}_{s-}|^{2}\vee|Y^{n,m}_{s-}+U^{n,m}_s(e)|^2\right)^{\frac{p-2}{2}}\mathds{1}_{\{|Y^{n,m}_{s-}|\vee|Y^{n,m}_{s-}+|U^{n,m}_s(e)|\neq 0\}}|U^{n,m}_s(e)|^2\mu(ds,de)\right] .
			\end{split}
		\end{equation*}
		Then, after taking the supremum in (\ref{eq6}) and then the expectation, we get for $\rho>2$
		\begin{equation}\label{eq7}
			\mathbb{E}\left[\sup_{0\leq t\leq T}|Y_{t}^{n}-Y^m_{t}|^p\right]\xrightarrow[n,m\to +\infty]{}0.
		\end{equation}
		
		It follows that $\{Y^n\}_{n\geq0}$ is a Cauchy sequence in $\mathcal{S}^{p}$. Since $Y^m\nearrow Y$ then $\mathbb{E}\big[\sup\limits_{0\leq t\leq T}|Y_{t}^{n}-Y_{t}|^p\big]\xrightarrow[n\to +\infty]{}0$ and $Y\in\mathfrak{B}^{p}_\beta$ (from \textbf{Step 1}). 
		
		Thank again to the argumentation used in Proposition \ref{Estimation p less stricly than 2}, wihc ny reporforming those, we get
		\begin{eqnarray*}
			&&\mathbb{E}\left[ \left(\int_t^T|Z_{s}^{n,m}|^2ds\right)^\frac{p}{2}\right] +\mathbb{E}\left[ \left(\int_t^T\|U^{n,m}_s\|^2_\lambda ds\right)^\frac{p}{2}\right] +\mathbb{E}\left[ \left(\int_0^{T}\int_{\mathcal{U}}|U^{n,m}_s(e)|^2\mu(ds,de)\right)^{\frac{p}{2}}\right] \nonumber\\
			&\leq&\mathfrak{C}_{p,\beta,\epsilon} \left( \mathbb{E}\left[\sup_{0\leq t\leq T}|Y_{t}^{n,m}|^{p}\right]+\mathbb{E}\int_0^{T}|Y_{s}^{n,m}|^{p-2}|Z_{s}^{n,m}|^2\mathds{1}_{\{Y^n_s\neq Y^m_s\}}ds\right. \nonumber\\
			&&\left.\quad +\mathbb{E}\int_0^T\int_{\mathcal{U}}\left(|Y^{n,m}_{s-}|^{2}\vee|Y^{n,m}_{s-}+U^{n,m}_s(e)|^2\right)^{\frac{p-2}{2}}\mathds{1}_{\{|Y^{n,m}_{s-}|\vee|Y^{n,m}_{s-}+U^{n,m}_s(e)|\neq 0\}}|U^{n,m}_s(e)|^2\mu(ds,de)\right) .
		\end{eqnarray*}
		Then, from (\ref{CV}) and (\ref{eq7}) we get
		\begin{equation*}
			\begin{split}
				&\mathbb{E}\left[ \left(\int_0^T|Z_{s}^{n,m}|^2ds\right)^\frac{p}{2}\right] +\mathbb{E}\left[ \left(\int_0^T\|V^{n,m}_s\|^2_\lambda ds\right)^\frac{p}{2}\right] +\mathbb{E}\left[ \left(\int_0^{T}\int_{\mathcal{U}}|U^{n,m}_s(e)|^2\mu(ds,de)\right)^{\frac{p}{2}}\right] \\
				&\xrightarrow[n,m\to +\infty]{}0.
			\end{split}
		\end{equation*}
		Then, $\{Z^n,U^n\}_{n\geq0}$ is a Cauchy sequence of processes in $\mathcal{H}^p\times\mathfrak{L}^p_{\lambda}$ and then there exists a pair of processes $(Z,V)$ such that the sequences $\{Z^n\}_{n\geq0}$ and $\{U^n\}_{n\geq0}$ converge toward $Z\in\mathcal{H}^p$ and $U\in\mathfrak{L}^p_\lambda$ respectively.
		To conclude, from (\ref{penalized}), we have
		$$K^n_t=Y^n_t-Y^n_0+\int_0^t \mathfrak{f}(s, Y^n_s)ds ds-\int_0^tZ^n_sdB_s-\int_0^t\int_{\mathcal{U}}V^n_s(e)\tilde{\mu}(ds,de).$$
		Then
		$$
		\mathbb{E}\left[\sup_{0\leq t\leq T}|K_{t}^{n}-K_{t}^m|^{p}\right]\xrightarrow[n,m\to +\infty]{}0.
		$$
		It follows that $\{K^n\}_{n\geq0}$ is a Cauchy sequence in $\mathcal{S}^p$. Therefore, there exists an $\mathbb{F}$-predictable ($K$ is equal to its dual predictable projection) process $K\in\mathcal{K}^p$  and
		$$\mathbb{E}\left[\sup_{0\leq t\leq T}|K_{t}^{n}-K_{t}|^{p}\right]\xrightarrow[n\to +\infty]{}0.
		$$
		
		Finally, since the sequence $\{Y^n\}_{n \geq 0}$ is increasing and consists of RCLL processes, and the sequence $\{K^n\}_{n \geq 0}$ is non-decreasing and predictable with $K_0^n = 0$, the limit process $(K_t)_{t \leq T}$ is also a non-decreasing predictable process with $K_0 = 0$. Therefore, by the monotonic limit theorem due to \cite[Lemma 2.2]{peng}, the processes $Y$ and $K$ are RCLL. Furthermore, from \textbf{Step 1}, we also deduce that $(Z, U) \in \mathcal{H}^p_\beta \times \mathfrak{L}^p_{\lambda, \beta}$.
	\end{proof}		
	
	\textbf{Step 4:} The limiting process $(Y_t,Z_t,U_t,K_t)_{t\leq T}$ is the solution of the RBSDEJs \eqref{basic equation} associated with $(\xi,\mathfrak{f},L)$. 
	
	\begin{proof}
		By passing to the limits in the approximate BSDEJs (\ref{penalized}), we obtain that the quadruple of processes $(Y,Z,U,K)$ satisfy the following discontinuous BSDE:
		$$Y_t=\xi+\int_t^T\mathfrak{f}(s,Y_s)ds+K_T-K_t-\int_t^TZ_sdB_s-\int_t^T\int_{\mathcal{U}}U_s(e)\tilde{\mu}(ds,de)\quad \forall t \in [0,T].$$
		Recall that in \textbf{Step 2}, we had proved that $Y_t\geq L_t$ a.s. $\forall t \in [0,T]$.\\
		To conclude, it remains to show the Skorokhod's condition. Indeed, set $L_{t}^{\xi}=L_{t}\mathds{1}_{\{t<T\}}+\xi\mathds{1}_{\{t=T\}}$. By using the notion of Snell envelopes ${\bf Sn}(.)$, we know that
		$$Y_{t}-\mathbb{E}\left[\xi +\int_{t}^{T}\mathfrak{f}(s,Y_s)ds|\mathcal{F}_{t}\right]={\bf Sn}(\pi_t),\quad\mbox{ where }\quad\pi_t=L_{t}^{\xi}-\mathbb{E}\left[\xi+\int_{t}^{T}\mathfrak{f}(s,Y_s)ds|\mathcal{F}_{t}\right],\quad\pi_T=0 .$$
		Since ${\bf Sn}(\pi)$ is a strong supermartingale of class (D), with the Doob-Meyer decomposition theorem (see Theorem 8 in \cite[p. 111]{Protter2005}), there exists a unique uniformly integrable $\mathbb{F}$-martingale $M$ and a unique $\mathbb{F}$-adapted RCLL increasing processes $\mathcal{K}$ with $\mathbb{E}[\mathcal{K}_T]<+\infty$ and $\mathcal{K}_0=0$ such that ${\bf Sn}(\pi_t)=M_t-\mathcal{K}_t$. We stress that the uniqueness of solution implies that $\mathcal{K}=K$. Further, from Proposition 5.2 in \cite{Ess}, we have for any $t \in [0,T]$
		$$\Delta K_t^d=({\bf Sn}(\pi_{t})-\pi_{t-})^-\mathds{1}_{\{({\bf Sn}(\pi_{t-})=\pi_{t-})\}}=(Y_t-L_{t-})^-\mathds{1}_{\{Y_{t-}=L_{t-}\}}.$$
		On the other hand, remark that the process $\left({\bf Sn}(\pi_t)+K^d_t=M_t-K^c_t\right)_{t\leq T}$ is the Snell envelope of $\left(\pi_t+K^d_t\right)_{t\leq T}$ which is regular (i.e, ${}^p({\bf Sn}(\pi_t)+K^d_t)={\bf Sn}(\pi_{t-})+K^d_{t-}$). Then, for any $t\in [0,T]$, the stopping time $\tau_t:=\inf\{s\geq t,\;\; K_s^c>K_t^c\}\wedge T$ is optimal in the class of stopping times (see Theorem 2.41 page 140 in \cite{Elk}),
		therefore, we have
		$$\int_{t}^{\tau_t}({\bf Sn}(\pi_{t})+K^d_t-(\pi_{t}+K^d_t))dK^c_{t}=\int_{t}^{\tau_t}(Y_{t}-L_{t})dK^c_{t}=0.$$
		Consequently, $\int_{0}^{T}(Y_{t}-L_{t})dK^c_{t}=0$.	
	\end{proof}

	Theorem \ref{lem2} is then proved.
	
\end{proof}

From Theorem \ref{lem2}, we may derive the following Corollary:
\begin{corollary}\label{help}
	Assume that $(\mathcal{H}1)$--$(\mathcal{H}3)$ hold for a sufficient large $\beta$. If $(z,u)\in\mathcal{H}^p_\beta\times\mathfrak{L}^p_{\lambda,\beta}$ for any $p\in(1,2)$, then there exists a unique processes $(Y,Z,U,K)\in\mathcal{E}^p_\beta$ solution to the following RBSDEJs:
	\begin{eqnarray}\label{helpbsde}
		&&\hspace{-1cm} \displaystyle\left\{
		\begin{array}{ll}
			Y_{t}=\xi +\displaystyle\int_{t}^{T} f(s,Y_{s},z_{s},u_s)ds+(K_{T}-K_{t})-\displaystyle\int_{t}^{T}Z_{s}dB_{s}-\int_t^T\int_{\mathcal{U}}U_s(e)\tilde{\mu}(ds,de)& \hbox{}\\
			Y_t\geq L_t, \quad \forall t\leq T & \hbox{}\\
			\displaystyle\int_0^T(Y_t-L_t)dK^c_t=0  \quad \mbox{and} \quad  \Delta K_t^d=(Y_t-L_{t-})^-\mathds{1}_{\{Y_{t-}=L_{t-}\}}. &\hbox{}
		\end{array}
		\right.
	\end{eqnarray}
\end{corollary}
\begin{proof}
	For every $n \geq 1$, we define
	$$
	\xi^n := q_n(\xi), \quad \mathfrak{f}_n(\omega, t, y) := f(\omega, t, y, z_t, u_t) - f(\omega, t, 0, z_t, u_t) + q_n(f(\omega, t, 0, z_t, u_t)),
	$$
	with $q_n(x) = \frac{x n}{|x| \vee n}$. Clearly, for any $y \in \mathbb{R}^d$, the process $\mathfrak{f}_k(\cdot, y)$ is progressively measurable. Moreover, from Remark \ref{rmq essential}, the driver $\mathfrak{f}_k$ satisfies, for all $t \in [0,T]$, $y, y' \in \mathbb{R}^d$, and $d\mathbb{P} \otimes dt$-a.e.,
	\begin{equation}\label{fk new mon}
		\left(y - y'\right)\left(\mathfrak{f}_k(t, y) - \mathfrak{f}_k(t, y')\right) \leq 0.
	\end{equation}
	Furthermore, it is evident that $\big|q_n(\xi)\big| \leq n$ and 
	$$
	\big|\mathfrak{f}_n(\omega, t, 0)\big|^p = \big|q_n(f(\omega, t, 0, z_t, u_t))\big|^p \leq n^p,
	$$
	which implies
	$$
	\mathbb{E}\int_{0}^{T} e^{\beta A_s} \big|\mathfrak{f}_n(t, 0)\big|^p \, ds \leq n^p \mathbb{E}\int_{0}^{T} e^{\beta A_s} \big|\varphi_s\big|^p \, ds < +\infty.
	$$
	Using Theorem \ref{lem2}, we deduce that, for each $n \geq 1$, there exists a unique triplet $(Y^n, Z^n, U^n,K^n)$ that belongs to $\mathcal{E}^p_\beta$ and solves the following BSDE:
	\begin{eqnarray}\label{Dhr}
		&&\hspace{-1cm} \displaystyle\left\{
		\begin{array}{ll}
			Y^n_{t}=\xi^{n} +\displaystyle\int_{t}^{T} \mathfrak{f}_n(s,Y^n_{s})ds+(K^n_{T}-K^n_{t})-\displaystyle\int_{t}^{T}Z^n_{s}dB_{s}-\int_t^T\int_{\mathcal{U}}U^n_s(e)\tilde{\mu}(ds,de)& ~t \in [0,T].\\
			Y^n_t\geq L_t, \quad \forall t \in [0,T] & \hbox{}\\
			\displaystyle\int_0^T(Y^n_{t-}-L_{t-})dK^{n}_t=0. 
		\end{array}
		\right.
	\end{eqnarray}
	Let us fix integers $ m >n \geq 1$ such that $k \leq m$. Let $(Y^n, Z^n, U^n, K^n)$ and $(Y^m, Z^m, U^m, K^m)$ be two $\mathbb{L}^p$-solutions of the BSDE \eqref{Dhr} associated with the parameters $(\xi^n, \mathfrak{f}_n,L)$ and $(\xi^m, \mathfrak{f}_m, L)$, respectively. Define ${\mathcal{R}}^{n,m} = \mathcal{R}^n - \mathcal{R}^m$, where $\mathcal{R} \in \{\xi, Y, Z, U\}$. Our goal is to show that the sequence $\left\{\left(Y^n, Z^n, U^n, K^n\right)\right\}_{n \geq 1}$ is a Cauchy sequence in $\mathcal{E}^p_\beta$. To this end, we first provide estimations for the generator part of the semimartingale $\big(e^{\frac{p}{2}\beta A_t} \big|{Y}^{n,m}_t\big|\big)_{t \leq T}$. Using Young's inequality, we have
	\begin{equation}\label{new--generator}
		\begin{split}
			& p e^{\frac{p}{2}\beta A_s} \big|{Y}^{n,m}_s\big|^{p-1} {{\hat{Y}}}^{n,m}_s
			\left(\mathfrak{f}_n(s, Y^n_s) - \mathfrak{f}_m(s, Y^m_s)\right) \, ds \\
			& \leq p e^{\frac{p}{2}\beta A_s} \mathds{1}_{\{\widehat{Y}^{n,m}_s \neq 0\}} \big|{Y}^{n,m}_s\big|^{p-1} \left|q_n(f(s, 0, z_s, u_s)) - q_m(f(s, 0, z_s, u_s))\right| \, ds \\
			& \leq p \left(e^{\frac{p-1}{2}\beta A_s} \big|{Y}^{n,m}_s\big|^{p-1} a_s\right) \left(e^{\frac{1}{2}\beta A_s} \frac{\left|q_n(f(s, 0, z_s, u_s)) - q_m(f(s, 0, z_s, u_s))\right|}{a_s}\right) \, ds \\
			& \leq (p-1) e^{\frac{p}{2}\beta A_s} \big|{Y}^{n,m}_s\big|^p \, dA_s + e^{\frac{p}{2}\beta A_s} \left|\frac{\left|q_n(f(s, 0, z_s, u_s)) - q_m(f(s, 0, z_s, u_s))\right|}{a_s}\right|^p \, ds.
		\end{split}
	\end{equation}
	Moreover as in the Skorokhod condition \eqref{Skoro1} this time for $Y^n$ and $Y^m$ related to $K^n$ and $K^m$, we get for any $t \in [0,T]$
	\begin{equation*}
		\begin{split}
			&\int_t^T e^{\frac{p}{2} \beta A_s} |{Y}^{n,m}_{s-}|^{p-1} \hat{{Y}}^{n,m}_{s-} \, d{K}^{n,m}_s  \leq 0.
		\end{split}
	\end{equation*}
	Then, following similar arguments as those used in Proposition \ref{Estimation p less stricly than 2}, with \eqref{new--generator} and an appropriate choice of $\beta$, we deduce the existence of a constant $ \mathfrak{C}_{p,\beta,\epsilon}$ such that
	\begin{equation}\label{UI p less than 2}
		\begin{split}
			&\mathbb{E}\left[\sup_{t \in [0,T]} e^{\frac{p}{2}\beta A_{t  }}\big|\widehat{Y}^{n,m}_{t}\big|^p \right]+\mathbb{E}\left[ \int_{0}^{T}e^{\frac{p}{2}\beta A_{s}}\big|{Y}^{n,m}_{s}\big|^p dA_s\right] +\mathbb{E}\left[\left(\int_{0}^{T}e^{\beta A_s} \big\|{Z}^{n,m}_s\big\|^2 ds\right)^{\frac{p}{2}}\right]\\
			+&\mathbb{E}\left[\left(\int_{0}^{T}e^{\beta A_s}\int_{\mathcal{U}} \big|{U}^{n,m}_s(e)\big|^2 \mu(ds,de)\right)^{\frac{p}{2}}\right]+\mathbb{E}\left[\left(\int_{0}^{T}e^{\beta A_s} \big\|{U}^{n,m}_s\big\|^2_{\mathbb{L}^2_\lambda} ds\right)^{\frac{p}{2}}\right]\\
			&\leq   \mathfrak{C}_{p,\beta,\epsilon} \left(\mathbb{E}\left[ e^{\frac{p}{2}\beta A_T}\big|{\xi}^{n,m}\big|^p\right] +\mathbb{E}\left[ \int_{0}^{T}e^{\frac{p}{2}\beta A_s} \left|\frac{\left|q_n(f(s,0,z_s,u_s))-q_m(f(s,0,z_s,u_s))\right|}{a_s}\right|^p ds\right] \right).
		\end{split}
	\end{equation}
	On the other hand, by using the basic inequality \eqref{basic inequality} and assumption $(\mathcal{H}2)$-(iii)-(iv) on $f$, we have
	\begin{equation*}
		\begin{split}
			\left|\frac{\left|q_n(f(s,0,z_s,u_s))-q_m(f(s,0,z_s,u_s))\right|}{a_s}\right|^p
			&\leq2^p\left(\left|\frac{\left|q_n(f(s,0,z_s,u_s))\right|}{a_s}\right|^p+\left|\frac{\left|q_m(f(s,0,z_s,u_s))\right|}{a_s}\right|^p\right)\\
			&\leq 2^{2p+1} \left(\left|\frac{\left|f(s,0,z_s,u_s)-f(s,0,0,0)\right|}{a_s}\right|^p+\left|\frac{\varphi_s}{a_s}\right|^p\right)\\
			& \leq \mathfrak{C}_p\left(|z_s|^p+\|u_s\|^p_{\mathbb{L}^2_\lambda}+\left|\frac{\varphi_s}{a_s}\right|^p\right)
		\end{split}
	\end{equation*}
	Then, passing to the Lebesgue-Stieltjes integral and applying Jensen's inequality, we obtain
	\begin{equation}\label{Generator p less than 2}
		\begin{split}
			&\int_{0}^{T}e^{\frac{p}{2}\beta A_s} \left|\frac{\left|q_n(f(s,0,z_s,u_s))-q_m(f(s,0,z_s,u_s))\right|}{a_s}\right|^p ds\\
			& \leq  \mathfrak{C}_{p,\beta,\epsilon,T}\left(  \left(\int_{0}^{T}e^{\beta A_s}\big\|z_s\big\|^2 ds \right)^{\frac{p}{2}}+ \left(\int_{0}^{T}e^{\beta A_s}\big\|u_s\big\|^2_{\mathbb{L}^2_Q}ds\right)^{\frac{p}{2}}+ \int_{0}^{T}e^{\beta A_s}\left|\varphi_s\right|^p ds\right) .
		\end{split}
	\end{equation}
	From the definitions of $q_n$ and $q_m$, we have $\lim\limits_{n,m \rightarrow +\infty} \left|q_n(f(s,0,z_s,u_s)) - q_m(f(s,0,z_s,u_s))\right| = 0$ a.s. Then, using the estimation above, along with the facts that $z \in \mathcal{H}^2_\beta$, $u \in \mathfrak{L}^p_{\beta}$, assumption \textsc{(H4)}, and the Lebesgue dominated convergence theorem, we conclude that $\left\{\left(Y^n,Z^n,U^n\right)\right\}_{n \geq 1}$ is a Cauchy sequence in $\mathcal{E}^p_\beta(0,T)$. Thus, there exists a triplet of stochastic processes $(Y,Z,U)$ that represents the limiting process of the sequence $\left\{(Y^n,Z^n,U^n)\right\}_{n \geq 1}$ in the space $\mathcal{E}^p_\beta(0,T)$. By the BDG inequality, we have
	\begin{equation*}
		\begin{split}
			\lim\limits_{n \rightarrow+\infty}\mathbb{E}\left[\sup_{0\leq t\leq T} \left|\int_{t}^{T}Z^n_s dW_s-\int_{t}^{T}Z_s dW_s\right|^p\right] \leq \mathfrak{c}\lim\limits_{n \rightarrow+\infty}\mathbb{E}\left[\left(\int_{0}^{T}\big\|Z^n_s-Z_s\big\|^2ds\right)^{\frac{p}{2}}\right] =0.
		\end{split}
	\end{equation*}
	Similarly, we derive
	\begin{equation*}
		\begin{split}
			\lim\limits_{n \rightarrow+\infty}\mathbb{E}\left[\sup_{0\leq t\leq T} \left|\int_{t}^{T} \int_{\mathcal{U}} U^n_s(e) \tilde{\mu}(ds,de)-\int_{t}^{T}\int_E U_s(e) \tilde{\mu}(ds,de)\right|^p\right]=0.
		\end{split}
	\end{equation*}
	Note that $\big|q_n(\xi)\big|^p \leq \big|\xi\big|^p$ and $\lim\limits_{n \rightarrow +\infty} q_n(\xi) = \xi$ a.s. Then, using assumption $(\mathcal{H}1)$ and the dominated convergence theorem, we obtain 
	$$
	\lim\limits_{n \rightarrow +\infty} \mathbb{E}\left[e^{\frac{p}{2}\beta A_T} \big|q_n(\xi) - \xi\big|^p\right] = 0.
	$$
	Next, applying the estimation \eqref{UI p less than 2} and passing to the limit as $n \rightarrow +\infty$, along with Fatou's Lemma and estimation \eqref{Generator p less than 2}, we obtain
	\begin{equation}\label{UI p less than 2--limmiting}
		\resizebox{\textwidth}{!}{$
			\begin{split}
				&\mathbb{E}\left[\sup_{t \in [0,T]} e^{\frac{p}{2}\beta A_{t  }}\big|{Y}_{t}\big|^p \right]+\mathbb{E}\left[ \int_{0}^{T}e^{\frac{p}{2}\beta A_{s}}\big|{Y}_{s}\big|^p dA_s\right] +\mathbb{E}\left[\left(\int_{0}^{T}e^{\beta A_s} \big\|{Z}_s\big\|^2 ds\right)^{\frac{p}{2}}\right]\\
				&+\mathbb{E}\left[\left(\int_{0}^{T}e^{\beta A_s}\int_{\mathcal{U}} \big|{U}_s(e)\big|^2 \mu(ds,de)\right)^{\frac{p}{2}}\right]+\mathbb{E}\left[\left(\int_{0}^{T}e^{\beta A_s} \big\|{U}_s\big\|^2_{\mathbb{L}^2_\lambda} ds\right)^{\frac{p}{2}}\right]\\
				&\leq   \mathfrak{C}_{p,\beta,\epsilon,T} \left(\mathbb{E}\left[ e^{\frac{p}{2}\beta A_T}\big|{\xi}\big|^p\right] +\mathbb{E}\left[\left(\int_{0}^{T}e^{\beta A_s}\big\|z_s\big\|^2 ds \right)^{\frac{p}{2}}\right] +\mathbb{E}\left[\left(\int_{0}^{T}e^{\beta A_s}\big\|u_s\big\|^2_{\mathbb{L}^2_\lambda} ds\right)^{\frac{p}{2}}\right]+\mathbb{E}\left[ \int_{0}^{T}e^{\beta A_s}\left|\varphi_s\right|^p ds\right] \right).
			\end{split}
			$}
	\end{equation}
	Using the continuity of the function $f$ w.r.t. to the $y$-variable (assumption \textsc{(H6)}), we deduce that 
	$$
	\lim\limits_{n \rightarrow +\infty} \mathfrak{f}_n(t,Y^n_t) = f(t,Y_t,z_t,u_t) \quad \text{a.s.}
	$$
	Additionally, by reporfoming similar computations as in \eqref{K.3} and using the Lebesgue dominated convergence theorem, we can derive
	$$
	\lim\limits_{n \rightarrow +\infty}\mathbb{E}\left[\left(\int_{0}^{T}\left|\mathfrak{f}_n(s,Y^n_s)- f(s,Y_s,z_s,u_s)\right| ds\right)^p\right]=0.
	$$
	
	To conclude, from (\ref{Dhr}), we have
	$$K^n_t=Y^n_t-Y^n_0+\int_0^t \mathfrak{f}(s, Y^n_s) ds-\int_0^tZ^n_sdB_s-\int_0^t\int_{\mathcal{U}}V^n_s(e)\tilde{\mu}(ds,de).$$
	Then, using the previous convergence results (as in Theorem \ref{lem2}), we get
	$$
	\mathbb{E}\left[\sup_{0\leq t\leq T}|K_{t}^{n}-K_{t}^m|^{p}\right]\xrightarrow[n,m\to +\infty]{}0.
	$$
	It follows that $\{K^n\}_{n\geq0}$ is a Cauchy sequence in $\mathcal{S}^p$. Therefore, there exists an $\mathbb{F}$-predictable process $K\in\mathcal{K}^p$  and
	$$\mathbb{E}\left[\sup_{0\leq t\leq T}|K_{t}^{n}-K_{t}|^{p}\right]\xrightarrow[n\to +\infty]{}0.
	$$
	Finally, Passing to the limit term by term in \eqref{Dhr}, we deduce that $(Y,Z,U,K)$ is the unique $\mathbb{L}^p$-solution of the RBSDEJs \eqref{helpbsde}. Completing the proof.

\end{proof}

\subsection{The general case where the driver $f$ depends on $(z,u)$}
Let us announce the main result of this section:
\begin{theorem}\label{t1}
	Assume that $(\mathcal{H}1)$--$(\mathcal{H}3)$ hold for a sufficient large $\beta$. Then, the RBSDEJs \eqref{basic equation} associated with parameters $(\xi,f,L)$ has a unique $\mathbb{L}^p$-solution for any $p\in(1,2)$.
\end{theorem}

\begin{proof}
	Let $(y,z,v)\in\mathcal{S}^{p,A}_\beta\times\mathcal{H}^p_\beta\times\mathfrak{L}^p_{\lambda,\beta}$ and we define $(Y,Z,V)=\Psi(y,z,v)$ where $(Y,Z,V,K)$ is the unique $L^p$-solution of the RBSDEJs (\ref{helpbsde}). For another element $(y',z',v')\in\mathcal{S}^{p,A}_\beta\times\mathcal{H}^p_\beta\times\mathfrak{L}^p_{\lambda,\beta}$, we define in the same way $(Y',Z',V')=\Psi(y',z',v')$ where $(Y',Z',V',K')$ is the unique $\mathbb{L}^p$-solution of the RBSDEJs associated with parameters $(\xi,f(\cdot,\cdot,z',v'),L)$. We denote $\bar{\Re}=\Re-\Re'$ for $\Re\in\{Y,Z,V,K,Y',Z',V',K'\}$. We are going to prove that the mapping $\Psi$ is a contraction strict on $\mathcal{S}^{p,A}_\beta\times\mathcal{H}^p_\beta\times\mathfrak{L}^p_{\lambda,\beta}$ equipped with the norm
	$$\|(Y,Z,V\|_{\mathcal{S}^{p,A}_\beta\times\mathcal{H}^{p}_\beta\times\mathfrak{L}^{p}_{\lambda,\beta}}^p
	:=\|Y\|_{\mathcal{S}^{p,A}_\beta}^p+\|Z\|_{\mathcal{H}^{p}_\beta}^p+\|V\|_{\mathfrak{L}^{p}_{\lambda,\beta}}^p.$$
	Actually, by applying the Corollary \ref{cor} in combination with Lemma \ref{lem1}, we have
	\begin{eqnarray}\label{eq1}
		&&e^{\frac{p}{2}\beta A_t}|\bar{Y}_t|^p+\frac{p}{2}\beta\int_t^Te^{\frac{p}{2}\beta A_s}|\bar{Y}_s|^pdA_s+c(p)\int_t^Te^{\frac{p}{2}\beta A_s}|\bar{Y}_s|^{p-2}|\bar{Z}_s|^2\mathds{1}_{\{\bar{Y}_s\neq0\}}ds\nonumber\\
		&&+c(p)\int_t^Te^{\frac{p}{2}\beta A_s}\int_{\mathcal{U}}|\bar{U}_s(e)|^2\left(|\bar{Y}_{s-}|^{2}\vee|\bar{Y}_{s-}+\bar{U}_s(e)|^2\right)^{\frac{p-2}{2}}\mathds{1}_{\{|\bar{Y}_{s-}|\vee|\bar{Y}_{s-}+\bar{U}_s(e)|\neq 0\}}\mu(ds,de)\nonumber\\
		&\leq&p\int_t^Te^{\frac{p}{2}\beta A_s}|\bar{Y}_s|^{p-1}\hat{\bar{Y}}_{s}\left(f(s,Y_s, z_s,u_s)-f(s,Y'_s, z'_s,u'_s)\right)ds+p\int_t^Te^{\frac{p}{2}\beta A_s}|\bar{Y}_{s-}|^{p-1}\hat{\bar{Y}}_{s-}d\bar{K}_s\nonumber\\
		&&-p\int_t^Te^{\frac{p}{2}\beta A_s}|\bar{Y}_s|^{p-1}\hat{\bar{Y}}_s\bar{Z}_sdB_s
		-p\int_t^T\int_{\mathcal{U}}e^{\frac{p}{2}\beta A_s}|\bar{Y}_{s-}|^{p-1}\hat{\bar{Y}}_{s-}\bar{U}_s(e)\tilde{\mu}(ds,de).
	\end{eqnarray}
	By employing a similar localization procedure as in Proposition \ref{Estimation p less stricly than 2} and taking the expectation on both sides of \eqref{eq1}, we get
	\begin{equation}\label{general case Ito less than 2}
		\begin{split}
			&\mathbb{E}\left[ e^{\frac{p}{2}\beta A_{t} }\big|\widehat{Y}_{t}\big|^p\right] +\frac{p}{2}\beta\mathbb{E}\left[ \int_{t  }^{T}e^{\frac{p}{2}\beta A_s} \big|\bar{Y}_s\big|^p dA_s\right] 
			+c(p)\mathbb{E}\left[ \int_{t}^{T}e^{\frac{p}{2}\beta A_s} \big|\bar{Y}_s\big|^{p-2}\big\| \bar{Z}_s \big\|^2 \mathds{1}_{\{\widehat{Y}_{s} \neq 0\}} ds\right] \\
			&+c(p)\mathbb{E}\left[ \int_t^Te^{\frac{p}{2}\beta A_s}\int_{\mathcal{U}}|\bar{U}_s(e)|^2\left(|\bar{Y}_{s-}|^{2}\vee|\bar{Y}_{s-}+\bar{U}_s(e)|^2\right)^{\frac{p-2}{2}}\mathds{1}_{\{|\bar{Y}_{s-}|\vee|\bar{Y}_{s-}+\bar{U}_s(e)|\neq 0\}}\lambda(de)ds\right] \\
			\leq& p\mathbb{E}\left[ \int_{t }^{T}e^{\frac{p}{2}\beta A_s} \big|\bar{Y}_s\big|^{p-1}\hat{\bar{Y}}_s  \left(f(s,Y_s, z_s,u_s)-f(s,Y'_s, z'_s,u'_s)\right)ds\right], 
		\end{split}
	\end{equation}	
	where we have used again the fact that 
	$$
	\int_t^T e^{\frac{p}{2}\beta A_s} |\bar{Y}_{s-}|^{p-1} \hat{\bar{Y}}_{s-} \, d\bar{K}_s \leq 0,
	$$
	in view of the identities \( dK_s = \mathds{1}_{\{Y_{s-} = L_{s-}\}} \, dK_s \) and \( dK'_s = \mathds{1}_{\{Y'_{s-} = L_{s-}\}} \, dK'_s \), together with the result from Lemma \ref{rem1}.
	
	Using assumptions $(\mathcal{H}2)$-(ii)-(iii) and Remark \ref{rmq essential}, we get 
	\begin{equation}\label{general case gene p less than 2}
		\begin{split}
			&\big|\bar{Y}_s\big|^{p-1}\hat{\bar{Y}_s}  \left(f(s,Y_s, z_s,u_s)-f(s,Y'_s, z'_s,u'_s)\right)\\
			&\leq \big|\bar{Y}_s\big|^{p-2}\mathds{1}_ {\{\bar{Y}_s \neq 0\}}\left(\alpha_s \big|\bar{Y}_s\big|^2+ \big|\bar{Y}_s\big| \big\|\bar{z}_s\big\|\delta_s+ \big|\bar{Y}_s\big| \big\|\bar{u}_s\big\|_{\mathbb{L}^2_\lambda} \eta_s \right)\\
			& \leq \alpha_s \big|\bar{Y}_s\big|^{p}  +\big|\bar{Y}_s\big|^{p-1} \big\|\bar{z}_s\big\|\delta_s+ \big|\bar{Y}_s\big|^{p-1} \big\|\bar{u}_s\big\|_{\mathbb{L}^2_\lambda}\eta_s \\
			& \leq \big|\bar{Y}_s\big|^{p-1} \big\|\bar{z}_s\big\|\delta_s+ \big|\bar{Y}_s\big|^{p-1} \big\|\bar{u}_s\big\|_{\mathbb{L}^2_\lambda}\eta_s .
		\end{split}
	\end{equation}
	On the other hand, using Hölder's inequality, Young's and Jensen's inequalities, for any $\varrho >0$, we obtain 
	\begin{equation*}
		\begin{split}
			\int_{t}^{T}e^{\frac{p}{2}\beta A_s }\big|\bar{Y}_s\big|^{p-1} \big|\bar{z}_s\big| \delta_s ds 
			&\leq  \left(\int_{t}^{T}e^{\frac{p}{2}\beta A_s }\big|\bar{Y}_s\big|^{p}dA_s \right)^{\frac{p-1}{p}}\left(\int_{t}^{T}e^{\frac{p}{2} \beta A_s} \big|\bar{z}_s\big|^p ds\right)^{\frac{1}{p}}\\
			&\leq \frac{(p-1)\varrho^{\frac{1}{p-1}}}{p} \int_{t}^{T}e^{\frac{p}{2}\beta A_s }\big|\bar{Y}_s\big|^{p}dA_s+\frac{(T-t)^{\frac{2-p}{2}}}{p\varrho} \left(\int_{t}^{T}e^{\beta A_s} \big\|\bar{z}_s\big\|^2 ds\right)^{\frac{p}{2}}.
		\end{split}
	\end{equation*}
	Similarly, we can show that 
	$$
	\int_{t}^{T}e^{\frac{p}{2}\beta A_s }\big|\bar{Y}_s\big|^{p-1} \big\|\bar{u}_s\big\|_{\mathbb{L}^2_\lambda} \eta_s ds \leq  \frac{(p-1)\varrho^{\frac{1}{p-1}}}{p} \int_{t}^{T}e^{\frac{p}{2}\beta A_s }\big|\bar{Y}_s\big|^{p}dA_s+\frac{(T-t)^{\frac{2-p}{2}}}{p\varrho} \left(\int_{t}^{T}e^{\beta A_s} \big\|\bar{u}_s\big\|^2_{\mathbb{L}^2_\lambda} ds\right)^{\frac{p}{2}}.
	$$
	Now, taking the expectation in the above two inequalities and applying Young's inequality, we obtain, for any $\varrho > 0$,
	\begin{equation}\label{genera well estimates}
		\begin{split}
			&\mathbb{E}\left[ \int_{t}^{T}e^{\frac{p}{2}\beta A_s }\big|\bar{Y}_s\big|^{p-1} \big|\bar{z}_s\big| \delta_s ds\right] +\mathbb{E}\left[ \int_{t}^{T}e^{\frac{p}{2}\beta A_s }\big|\bar{Y}_s\big|^{p-1} \big\|\bar{u}_s\big\|_{\mathbb{L}^2_\lambda} \eta_s ds\right] \\
			& \leq 2\frac{(p-1)\varrho^{\frac{1}{p-1}}}{p}\mathbb{E}\left[ \int_{t}^{T}e^{\frac{p}{2}\beta A_s }\big|\bar{Y}_s\big|^{p} dA_s\right] +\frac{(T-t)^{\frac{2-p}{2}}}{p\varrho} \left(\mathbb{E}\left[\left(\int_{t}^{T}e^{\beta A_s} \big|\bar{z}_s\big|^2 ds\right)^{\frac{p}{2}}\right]  \right.\\
			&\left.\qquad+\mathbb{E}\left[ \left(\int_{t}^{T}e^{\beta A_s} \big\|\bar{u}_s\big\|^2_{\mathbb{L}^2_\lambda} ds\right)^{\frac{p}{2}}\right] +\mathbb{E}\left[\left(\int_{t}^{T}e^{\beta A_s} \int_{\mathcal{U}} \big|\bar{u}_s(e)\big|^2 \mu(ds,de)\right)^{\frac{p}{2}}\right] \right).
		\end{split}
	\end{equation}
	Coming back to \eqref{general case Ito less than 2}, and using \eqref{general case gene p less than 2} and \eqref{genera well estimates}, we obtain
	\begin{equation*}
		\begin{split}
			&\mathbb{E}\left[ e^{\frac{p}{2}\beta A_{t} }\big|\bar{Y}_{t}\big|^p\right] +\frac{p}{2}\beta\mathbb{E}\left[ \int_{t  }^{T}e^{\frac{p}{2}\beta A_s} \big|\bar{Y}_s\big|^p dA_s\right] 
			+c(p)\mathbb{E}\left[ \int_{t}^{T}e^{\frac{p}{2}\beta A_s} \big|\bar{Y}_s\big|^{p-2}\big\| \bar{Z}_s \big\|^2 \mathds{1}_{\{\bar{Y}_{s} \neq 0\}} ds\right] \\
			&+c(p)\mathbb{E}\left[ \int_t^Te^{\frac{p}{2}\beta A_s}\int_{\mathcal{U}}|\bar{U}_s(e)|^2\left(|\bar{Y}_{s-}|^{2}\vee|\bar{Y}_{s-}+\bar{U}_s(e)|^2\right)^{\frac{p-2}{2}}\mathds{1}_{\{|\bar{Y}_{s-}|\vee|\bar{Y}_{s-}+\bar{U}_s(e)|\neq 0\}}\mu(ds,de)\right] \\
			&\leq 2(p-1)\varrho^{\frac{1}{p-1}} \mathbb{E}\left[ \int_{t}^{T}e^{\frac{p}{2}\beta A_s }\big|\widehat{Y}_s\big|^{p-1} dA_s\right] +\frac{(T-t)^{\frac{2-p}{2}}}{\varrho} \left(\mathbb{E}\left[ \int_{t}^{T}e^{\frac{p}{2}\beta A_s} \big|\bar{y}_s\big|^p dA_s\right] \right.\\
			&\left.\qquad+\mathbb{E}\left[\left(\int_{t}^{T}e^{\beta A_s} \big|\bar{z}_s\big|^2 ds\right)^{\frac{p}{2}}\right] +\mathbb{E}\left[\left(\int_{t}^{T}e^{\beta A_s} \big\|\bar{u}_s\big\|^2_{\mathbb{L}^2_\lambda} ds\right)^{\frac{p}{2}}\right]\right.\\
			&\left.\qquad+\mathbb{E}\left[\left(\int_{t}^{T}e^{\beta A_s} \int_{\mathcal{U}} \big|\bar{u}_s(e)\big|^2 \mu(ds,de)\right)^{\frac{p}{2}}\right] \right).
		\end{split}
	\end{equation*}
	It is worth noting that for any $\beta$ such that $\beta \geq 1  + 2(p-1)\varrho^{\frac{1}{p-1}}$, we deduce from the above estimation and a similar localization approach as in Proposition \ref{Estimation p less stricly than 2} that
	\begin{equation}\label{Using in global case}
		\begin{split}
			&\mathbb{E}\left[ \int_{0}^{T}e^{\frac{p}{2}\beta A_s} \big|\bar{Y}_s\big|^p dA_s\right] 
			+c(p)\mathbb{E}\left[ \int_{0}^{T}e^{\frac{p}{2}\beta A_s} \big|\bar{Y}_s\big|^{p-2}\big| \bar{Z}_s \big|^2 \mathds{1}_{\{\bar{Y}_{s} \neq 0\}} ds\right] \\
			&+ c(p)\mathbb{E}\left[ \int_0^Te^{\frac{p}{2}\beta A_s}\int_{\mathcal{U}}|\bar{U}_s(e)|^2\left(|\bar{Y}_{s-}|^{2}\vee|\bar{Y}_{s-}+\bar{U}_s(e)|^2\right)^{\frac{p-2}{2}}\mathds{1}_{\{|\bar{Y}_{s-}|\vee|\bar{Y}_{s-}+\bar{U}_s(e)|\neq 0\}}\mu(ds,de)\right] \\
			&+c(p)\mathbb{E}\left[ \int_0^Te^{\frac{p}{2}\beta A_s}\int_{\mathcal{U}}\left(|\bar{Y}_{s-}|^{2}\vee|\bar{Y}_{s-}+\bar{U}_s(e)|^2\right)^{\frac{p-2}{2}}\mathds{1}_{\{|\bar{Y}_{s-}|\vee|\bar{Y}_{s-}+\bar{U}_s(e)|\neq 0\}}|\bar{U}_s(e)|^2\lambda(de)ds\right] \\
			&\leq \frac{T^{\frac{2-p}{2}}}{\varrho}  \left(\mathbb{E}\left[ \int_{0 }^{T}e^{\frac{p}{2}\beta A_s} \big|\bar{y}_s\big|^p dA_s\right] +\mathbb{E}\left[\left(\int_{0}^{T}e^{\beta A_s} \big|\bar{z}_s\big|^2 ds\right)^{\frac{p}{2}}\right] +\mathbb{E}\left[ \left(\int_{0}^{T}e^{\beta A_s} \big\|\bar{u}_s\big\|^2_{\mathbb{L}^2_\lambda} ds\right)^{\frac{p}{2}}\right] \right).
		\end{split}
	\end{equation}
	Next, it suffices to repeat the arguments used in Proposition \ref{Estimation p less stricly than 2}, taking into account \eqref{Using in global case}. For the reader's convenience, we briefly outline them. Returning to \eqref{eq1} and applying the BDG inequality, we obtain, for any $\varrho' > 0$,
	\begin{equation*}
		\begin{split}
			&p\mathbb{E}\left[\sup_{t \in [0,T]}\left|\int_{t}^{T}e^{\frac{p}{2}\beta A_s} \big|\bar{Y}_s\big|^{p-1} \check{\bar{Y}}_s \bar{Z}_s dW_s\right|\right]\\
			&\leq \frac{1}{2 \varrho'} \mathbb{E}\left[\sup_{t \in [0,T]}e^{\frac{p}{2}\beta}\big|\bar{Y}_s\big|^{p}\right]+\frac{\varrho'}{2}p^2 \mathfrak{c}^2\mathbb{E}\left[\int_{0}^{T}e^{\frac{p}{2}\beta A_s} \big|\bar{Y}_s\big|^{p-2} \big|\bar{Z}_s\big|^2 \mathds{1}_{\{\bar{Y}_s \neq 0\}} ds\right].
		\end{split}
	\end{equation*}
	The last line follows from the basic inequality $ab \leq \frac{1}{2 \varrho} a^2 + \frac{\varrho}{2} b^2$ for any $\varrho > 0$ (recall that $\mathfrak{c} > 0$ is the universal constant in the BDG inequality). Similarly, by a comparable computation, for any $\varrho'' > 0$, we obtain
	\begin{equation*}
		\begin{split}
			&p\mathbb{E}\left[\sup_{t \in [0,T]}\left|\int_{t}^{T}e^{\frac{p}{2}\beta A_s}\int_{E} \big|\bar{Y}_{s-}\big|^{p-1} \hat{\bar{Y}}_{s-} \bar{U}_s(e)\tilde{\mu}(ds,de)\right|\right]\\
			& \leq \frac{1}{2\varrho''}\mathbb{E}\left[\sup_{0\leq t\leq T}e^{\frac{p}{2}\beta A_t}|\bar{Y}_{t}|^{p}\right]\\
			&\qquad+\frac{\varrho''}{2}p^2 \mathfrak{c}^2\mathbb{E}\left[ \int_0^Te^{\frac{p}{2}\beta A_s}\int_{\mathcal{U}}\left(|\bar{Y}_{s-}|^{2}\vee|\bar{Y}_{s-}+\bar{U}_s(e)|^2\right)^{\frac{p-2}{2}}\mathds{1}_{\{|\bar{Y}_{s-}|\vee|\bar{Y}_{s-}+\bar{U}_s(e)|\neq 0\}}|\bar{U}_s(e)|^2 \mu(ds,de)\right] .
		\end{split}
	\end{equation*}
	Plugging this into \eqref{eq1} and taking the supremum on both sides, we obtain
	\begin{equation*}
		\begin{split}
			&\left(1-\frac{1}{2}\left(\frac{1}{\varrho'}+\frac{1}{\varrho''}\right)\right)\mathbb{E}\left[\sup_{t \in [0,T]} e^{\frac{p}{2}\beta A_{t  }}\big|\bar{Y}_{t}\big|^p \right]\\
			\leq&  \frac{p^2 \mathfrak{c}^2}{2}\left(\frac{\varrho'}{c(p)}c(p)\mathbb{E}\left[\int_{0}^{T}e^{\frac{p}{2}\beta A_s} \big|\bar{Y}_s\big|^{p-2} \big|\bar{Z}_s\big|^2 \mathds{1}_{\{\bar{Y}_s \neq 0\}} ds\right]\right. \\
			&\left.\quad+\frac{\varrho''}{c(p)}c(p)\mathbb{E}\left[\int_0^Te^{\frac{p}{2}\beta A_s}\int_{\mathcal{U}}\left(|\bar{Y}_{s-}|^{2}\vee|\bar{Y}_{s-}+\bar{U}_s(e)|^2\right)^{\frac{p-2}{2}}\mathds{1}_{\{|\bar{Y}_{s-}|\vee|\bar{Y}_{s-}+\bar{U}_s(e)|\neq 0\}}|\bar{U}_s(e)|^2 \mu(ds,de)\right]\right).
		\end{split}
	\end{equation*}
	Choosing $\varrho'$ and $\varrho''$ such that $\frac{1}{2} \left(\frac{1}{\varrho'} + \frac{1}{\varrho''}\right) < 1$ and $\varrho' \vee \varrho'' < \mathfrak{b}_p$, we deduce the existence of a constant $\mathfrak{C}_p$ such that
	\begin{equation}\label{sup for p less than 2--global}
		\begin{split}
			&\mathbb{E}\left[\sup_{t \in [0,T]} e^{\frac{p}{2}\beta A_{t  }}\big|\bar{Y}_{t}\big|^p \right]\\
			\leq& \mathfrak{C}_p\left(\mathbb{E}\left[\int_{0}^{T}e^{\frac{p}{2}\beta A_s} \big|\bar{Y}_s\big|^{p-2} \big\|\widehat{Z}_s\big\|^2 \mathds{1}_{\{\bar{Y}_s \neq 0\}} ds\right] \right.\\
			&\left. ~ +\mathbb{E}\left[\int_0^T\int_{E}e^{\frac{p}{2}\beta A_s}\left(|\bar{Y}_{s-}|^{2}\vee|\bar{Y}_{s-}+\bar{U}_s(e)|^2\right)^{\frac{p-2}{2}}\mathds{1}_{\{|\bar{Y}_{s-}|\vee|\bar{Y}_{s-}+\bar{U}_s(e)|\neq 0\}}|\bar{U}_s(e)|^2 \mu(ds,de)\right]\right)\\
			&\leq \frac{\mathfrak{C}_p}{\varrho} \left(\mathbb{E}\int_{0 }^{T}e^{\frac{p}{2}\beta A_s} \big|\bar{y}_s\big|^p dA_s+	\mathbb{E}\left[\left(\int_{0}^{T}e^{\beta A_s} \big|\bar{z}_s\big|^2 ds\right)^{\frac{p}{2}}\right]  \right.\\
			&\left.\quad+	\mathbb{E}\left[\left(\int_{0}^{T}e^{\beta A_s} \big\|\bar{u}_s\big\|^2_{\mathbb{L}^2_\lambda} ds\right)^{\frac{p}{2}}\right] +\mathbb{E}\left[\left(\int_{t}^{T}e^{\beta A_s} \int_{\mathcal{U}} \big|\bar{u}_s(e)\big|^2 \mu(ds,de)\right)^{\frac{p}{2}}\right] \right).
		\end{split}
	\end{equation}
	From these estimations, and again using Proposition \ref{Estimation p less stricly than 2}, we obtain, for any $\varrho_\ast > 0$,
	\begin{equation*}
		\begin{split}
			&\mathbb{E}\left[\left(\int_{0}^{T}e^{\beta A_s}\int_{E} \big|\bar{U}_s(e)\big|^2 N(ds,de)\right)^{\frac{p}{2}}\right]\\
			& \leq \frac{2-p}{2}\varrho_\ast^{\frac{4}{p(2-p)}}\mathbb{E}\left[\sup_{t \in [0,T]} e^{\frac{p}{2}\beta A_{t  }}\big|\bar{Y}_{t}\big|^p\right]\\
			&\quad+\frac{p}{2 \varrho_\ast c(p)} c(p)\mathbb{E}\left[\int_{0}^{T}e^{\frac{p}{2}\beta A_s}\int_{\mathcal{U}}\big|\bar{U}_s(e)\big|^2\left(|\bar{Y}_{s-}|^{2}\vee|\bar{Y}_{s-}+\bar{U}_s(e)|^2\right)^{\frac{p-2}{2}}\mathds{1}_{\{|\bar{Y}_{s-}|\vee|\bar{Y}_{s-}+\bar{U}_s(e)|\neq 0\}} \mu(ds,de)\right],
		\end{split}
	\end{equation*}
	\begin{equation*}
		\begin{split}
			&\mathbb{E}\left[\left(\int_{0}^{T}e^{\beta A_s} \big\|\bar{U}_s\big\|^2_{\mathbb{L}^2_\lambda} ds\right)^{\frac{p}{2}}\right]\\
			& \leq \frac{2-p}{2}\varrho_\ast^{\frac{4}{p(2-p)}}\mathbb{E}\left[\sup_{t \in [0,T]} e^{\frac{p}{2}\beta A_{t  }}\big|\bar{Y}_{t}\big|^p\right]\\
			&\quad+\frac{p}{2 \varrho_\ast c(p)} c(p) \mathbb{E}\left[\int_{0}^{T}e^{\frac{p}{2}\beta A_s}\int_{\mathcal{U}}\left(|\bar{Y}_{s-}|^{2}\vee|\widehat{Y}_{s-}+\bar{U}_s(e)|^2\right)^{\frac{p-2}{2}}\mathds{1}_{\{|\bar{Y}_{s-}|\vee|\bar{Y}_{s-}+\bar{U}_s(e)|\neq 0\}}\big|\bar{U}_s(e)\big|^2 \lambda(de)ds\right],
		\end{split}
	\end{equation*}
	and
	\begin{equation*}
		\begin{split}
			\mathbb{E}\left[\left(\int_{0}^{T}e^{\beta A_s} \big|\bar{Z}_s\big|^2 ds\right)^{\frac{p}{2}}\right]
			&\leq \frac{2-p}{2}\varrho_\ast^{\frac{4}{p(2-p)}}\mathbb{E}\left[\sup_{t \in [0,T]} e^{\frac{p}{2}\beta A_{t}}\big|\widehat{Y}_{t}\big|^p\right]\\
			&\quad+\frac{p}{2 \varrho_\ast c(p)} c(p)\mathbb{E}\left[\int_{0}^{T}e^{\frac{p}{2 }\beta A_s}\big|\bar{Y}_{s}\big|^{p-2} \big|\bar{Z}_s\big|^2 \mathds{1}_{\{\bar{Y}_{s}\neq 0\} }ds\right].
		\end{split}
	\end{equation*}
	After summing the three terms above, it suffices to choose a suitable $\varrho_\ast$ such that $\frac{3(2-p)}{2}\varrho_\ast^{\frac{4}{p(2-p)}} + \frac{p}{2 \varrho_\ast \mathfrak{b}_p} < 1$, with the same choice of $\beta$ as above, to obtain
	\begin{equation}
		\begin{split}
			&\mathbb{E}\left[\int_{0}^{T}e^{\frac{p}{2}\beta A_s} \big|\bar{Y}_s\big|^p dA_s\right]+\mathbb{E}\left[\left(\int_{0}^{T}e^{\beta A_s} \big|\bar{Z}_s\big|^2 ds\right)^{\frac{p}{2}}\right]\\
			&+\mathbb{E}\left[\left(\int_{0}^{T}e^{\beta A_s}\int_{E} \big|\bar{U}_s(e)\big|^2 \mu(ds,de)\right)^{\frac{p}{2}}\right]+\mathbb{E}\left[\left(\int_{0}^{T}e^{\beta A_s} \big\|\widehat{U}_s\big\|^2_{\mathbb{L}^2_\lambda} ds\right)^{\frac{p}{2}}\right]\\
			&\leq \frac{\mathfrak{C}_{p,T,\epsilon}}{\varrho} \left(\mathbb{E}\int_{0 }^{T}e^{\frac{p}{2}\beta A_s} \big|\bar{y}_s\big|^p dA_s+\left(\int_{0}^{T}e^{\beta A_s} \big|\bar{z}_s\big|^2 ds\right)^{\frac{p}{2}} \right.\\
			&\left.\qquad+\left(\int_{0}^{T}e^{\beta A_s} \big\|\bar{u}_s\big\|^2_{\mathbb{L}^2_\lambda} ds\right)^{\frac{p}{2}}+\mathbb{E}\left[\left(\int_{0}^{T}e^{\beta A_s}\int_{E} \big|\bar{u}_s(e)\big|^2 \mu(ds,de)\right)^{\frac{p}{2}}\right]\right).
		\end{split}
	\end{equation}
	Finally, by choosing $\varrho$ large enough such that $\varrho > \mathfrak{C}_{p,T,\epsilon}$, we deduce that the functional $\Psi$ is a strict contraction mapping on $\mathcal{S}^{p,A}_\beta \times \mathcal{H}^p_\beta \times \mathfrak{L}^p_{\lambda,\beta}$. Hence, there exists a unique fixed point $(Y, Z, U)$ of $\Psi$ which, together with $K$, forms the unique $\mathbb{L}^p$-solution of the RBSDEJs \eqref{basic equation} associated with the parameters $(\xi, f, L)$. The proof is now complete.
\end{proof}


			\end{document}